\documentclass[11pt]{amsart}
\usepackage[mathscr]{eucal}
\usepackage{fullpage}
\usepackage{graphicx}
\usepackage{float}
\restylefloat{figure}

\newtheorem{theorem}{Theorem}[section]
\newtheorem{lemma}[theorem]{Lemma}
\newtheorem{proposition}[theorem]{Proposition}
\newtheorem{corollary}[theorem]{Corollary}
\newtheorem{assumption}{Assumption}
\newtheorem{remark}[theorem]{Remark}

\newtheorem{example}[theorem]{Example}

\numberwithin{equation}{section}

\def\NN{\hbox{I\kern-.2em\hbox{N}}}
\def\RR{\hbox{I\kern-.2em\hbox{R}}}

\date{}
\begin{document}

\title[Paper LB5]{On convergence of solutions to difference equations with additive perturbations}
\author[E. Braverman, and A. Rodkina]
{E. Braverman, and A. Rodkina}
\address{Department of Mathematics, University of Calgary, Calgary, Alberta T2N1N4, Canada}
\email{maelena@math.ucalgary.ca}
\address{Department of Mathematics,
The University of the West Indies, Mona Campus, Kingston, Jamaica}
\email{alexandra.rodkina@uwimona.edu.jm}

\begin{abstract}
Various types of stabilizing controls lead to a deterministic difference equation with the following property: once the initial 
value is positive, the solution tends to the unique positive equilibrium. 
Introducing additive  perturbations can change this picture: we give examples of difference equations experiencing additive 
perturbations which have solutions staying around zero rather than tending to the unique positive equilibrium. 
When perturbations are stochastic with a bounded support, we give an upper estimate for 
the probability that the solution  can stay around zero.  Applying extra conditions on the behavior of the map function $f$ at zero 
or on the amplitudes of stochastic perturbations,  we prove that the solution tends to the unique positive equilibrium almost surely. 
In particular, this holds either for all amplitudes when the right derivative of the map $f$ at zero exceeds one 
or, independently of the behavior of $f$ at zero, when the amplitudes are not square summable.


\end{abstract}

\maketitle

\section{Introduction}
\label{sec:Intr}

Consider the deterministic equation
\begin{equation}
\label{eq:ggammaintr}
x_{n+1}= \max\left\{ f(x_n)+\gamma_{n+1}, 0 \right\}, \quad x_0>0, \quad n=0,1, \dots,
\end{equation}
as well as the  stochastic difference equation with additive perturbations
\begin{equation}
\label{eq:stochintr}
x_{n+1}= \max\left\{ f(x_n)+\sigma_n\xi_{n+1}, 0 \right\}, \quad x_0>0, \quad  n=0,1, \dots.
\end{equation}
Here $f:\RR^+\to \RR^+$ is continuous, $f(0)=0$, and $f(x)>0$ for all $x>0$; $f$ has only 
two fixed points $x=0$ and $x=K>0$; $f(x)>x$ for all $0<x<K$ and $f(x)<x$ for all $x>K$. 
Deterministic perturbations $\gamma_n\in (-\infty, \infty)$ satisfy $\gamma_n\to 0$ as $n\to 0$, the noises $\xi_n$ are continuously 
distributed,  independent and bounded random variables with the joint support $[-1, 1]$, coefficients $\sigma_n>0$.

Difference equations $x_{n+1}=  f(x_n)$ have been an object of intensive attention since 1970ies,
their behavior can be unstable or even chaotic. Various methods have been introduced to 
stabilize difference equations, see the recent publications
\cite{BF15,BL,BRJDEA13,IMT13,LP14,Medina12,Shai} and references therein. 
In many cases a unique positive equilibrium is stabilized \cite{BF15,BL,IMT13}, it is also
common to stabilize a cycle \cite{BL,BRJDEA13,IMT13,LP14}. Difference equations which include memory (delay), higher order equations, or some other generalizations are stabilized in \cite{BF15,IMT13,Medina12}.  These are mostly deterministic equations, stochastic difference equations are considered in \cite{ ABR, AMR, AR1,  KR09, Shai}.

In the present paper, we consider stochastically perturbed deterministic equations. 
As illustrated in \cite{Brav2011},  introduction of stochastic perturbations in a population
dynamics model with the Allee effect can either eliminate the Allee zone, or bring all solutions to extinction,
in contrast to the behavior of the deterministic equation. In \cite{Brav2011} a random perturbation was assumed to take two
possible values, one positive and one negative. 
Here we suppose that the amplitudes $\sigma_n$ of random perturbations eventually  vanish. We explore whether such perturbations can 
change the eventual behavior of a stable difference equation.
Simple examples illustrate that deterministic perturbations decaying to zero can make the unstable zero equilibrium an attractor 
for solutions with small enough initial value. The purpose of the present paper is to explore conditions
under which stochastic perturbations with amplitudes tending to zero will not change stability of the positive equilibrium.
In the cases when preservation of stability cannot be claimed, we estimate its probability. 

Consider the difference equation
\begin{equation}
\label{eq:unpert}
x_{n+1}= f(x_n),  \quad x_0>0, \quad  n=0,1, \dots
\end{equation}
which is globally asymptotically stable in the positive domain.

Stochastic or even deterministic perturbations which tend to zero as $n \to \infty$ cannot make a point an attractor, if it is not an equilibrium  of the unperturbed equation. However, they can turn an unstable equilibrium into an attractor of some trajectories. 

We start with analyzing equations with deterministic perturbations: if instead of \eqref{eq:unpert} we consider perturbed equation 
\eqref{eq:ggammaintr} with $\gamma_n\to 0$, asymptotics of solutions can be different.
Theorem \ref{thm:detmain} states that the  solution $x_n$  of  deterministic equation 
\eqref{eq:ggammaintr} either tends  to $K$ or to zero. 
To illustrate the latter possibility,  we present examples of equations satisfying the assumptions of 
Theorem \ref{thm:detmain} but with solutions tending to zero.   
We  also discuss the connection between the function $f$  and  the perturbations $\gamma$ which guarantee that convergence to zero is 
impossible.

Further, we consider the asymptotics of the solution $x_n$ of stochastic difference equation \eqref{eq:stochintr} when we assume 
that, almost surely (a.s.),
$\displaystyle  \lim_{n\to \infty}\sigma_n\xi_{n+1}=0$.

Our main goal for the stochastic case is to prove that the probability of the eventual extinction 
\begin{equation}
\label{cond:intrprob0}
\mathbb P\{\omega\in \Omega: \lim_{n\to \infty} x_n(\omega)=0\}=0.
\end{equation}
Once this fact is verified, we can apply the results for a deterministic equation showing  that in this situation the only possibility 
left is $\displaystyle \lim_{n\to \infty}x_n=K$.
 
When we deal with stochastic equations we distinguish between the cases $\sigma\notin {\bf \ell}_2$ and $\sigma\in {\bf \ell}_2$.   
For  
$\sigma\notin {\bf \ell}_2$ we do not impose any extra assumptions on the behavior of $f$ at $x\to 0^+$.  Our main tools in this 
situation are  the Central Limit Theorem applied for  the sequence of uniformly bounded random variables  
(see  \cite[p. 328--333]{Shiryaev96}) and the result from \cite{AR1} which states that a. s.
\[
\limsup_{n\to \infty}\sum_{i=1}^n \sigma_i\xi_{i+1}=\infty, \quad 
\liminf_{n\to \infty}\sum_{i=1}^n \sigma_i\xi_{i+1}=\infty
\]
(see Section \ref{subsec:CLT} of Appendix and Lemma \ref{lem:+infty}). Armed with these results, we 
prove that \eqref{cond:intrprob0} holds.

In the case $\sigma\in {\bf \ell}_2$ we are able to prove \eqref {cond:intrprob0} when imposing some restriction on $f(x)$ as $x\to 
0^+$. In particular, we prove \eqref{cond:intrprob0} when
\begin{equation}
\label{cond:f0+}
\liminf_{x\to 0^+}~ \frac {f(x)}{x}>1,
\end{equation}
which is quite a common condition in population dynamics models. We also generalize condition  \eqref{cond:f0+}
assuming some connection between $\sigma$ and $f$ instead.
Without imposing any extra restrictions on $f$ as $x\to 0^+$ or  connection between $\sigma$ and $f$,  
for a symmetric distribution of $\xi$, we prove that 
$$
\mathbb P\{\omega\in \Omega: \lim_{n\to \infty} x_n(\omega)=0\}\le 1/2.
$$
Note that when $\sigma\in {\bf \ell}_2$,  we do not need any assumption about the expectations  $m_i=\mathbb E\xi_i$, while for  
$\sigma\notin {\bf \ell}_2$ we prove \eqref{cond:intrprob0} when  $m_i$ are either nonnegative or negative $m_i$ are quickly decaying 
in the sense that  $m^-=(m^-_i)_{i\in \mathbb N}\in {\bf \ell}_2$, where $a^-= \min\{ a,0\}$.
 
The paper is organized as follows.
Section  \ref{sec:not} outlines notations and assumptions which will be used later.
In Section  \ref{sec:det} we show that when instead of \eqref{eq:unpert} we consider perturbed equation 
\eqref{eq:ggammaintr} with $\gamma_n\to 0$, asymptotics of solutions can be different. 
Theorem \ref{thm:detmain} proves that the  solution $x_n$  of  deterministic equation 
\eqref{eq:ggammaintr} either tends  to $K$ or to zero.  To illustrate the latter case  in Section \ref{subsec:detnottozero}  we 
present examples of equations satisfying the assumptions of Theorem \ref{thm:detmain} but with solutions tending to zero.   
In Section \ref{subsec:detnottozero} we also  derive conditions on 
$f$ and perturbation $\gamma$ which guarantee that convergence to zero for the  solution $x_n$  of perturbed deterministic 
equation \eqref{eq:ggammaintr} is impossible. In Proposition \ref{prop:invalpar} we also show that all solutions converge to the 
positive
equilibrium when the perturbation $\gamma$ is small and the initial value $x_0>\varepsilon$ for some $\varepsilon>0$.
Section \ref{sec:prel} involves auxiliary results for stochastic sequences. In Section  \ref{sec:stoch} we prove results on 
convergence of the solution $x_n$ of stochastic difference equation \eqref{eq:stochintr}
to either originally stable positive equilibrium $K$ or to zero. The purpose is to find conditions under which the latter case has the zero probability.
All solutions converge to $K$, a. s., if the perturbations amplitudes are not in ${\bf \ell}_2$. 
Our  results are illustrated with computer simulations in Section \ref{sec:sim}.  
Several   proofs are deferred to Appendix.

\section{Notations and main assumptions}
\label{sec:not}

Denote
$\mathbb N=\{1, 2, \dots\}$, ${\mathbb R}=(-\infty, \infty)$, $\RR^+=[0,\infty)$,
and $\displaystyle
a^+:=\max\{a, 0 \}$,  $\displaystyle a^-:=\min\{a, 0 \}$, for each $a\in \mathbb R$.
As usual, we use the symbol ${\bf \ell}_2$ for the space of real sequences $a=(a_n)_{n\in \mathbf  N}$: $a\in {\bf 
\ell}_2$ satisfying ~$\displaystyle \sum_{i=0}^\infty a_i^2<\infty$.

\begin{assumption} 
\label{as:g12}
Assume that the function $f$ satisfies the following conditions:
\begin{enumerate}
\item[(A1)] $f:\RR^+\to \RR^+$ is continuous, $f(0)=0$, and $f(x)>0$ for all $x>0$;
\item[(A2)] $f$ has only two fixed points $x=0$ and $x=K>0$;\\
 $f(x)>x$ for all $0<x<K$ and $f(x)<x$ for all $x>K$.
\end{enumerate}
\end{assumption}

For $f$ satisfying Assumption \ref{as:g12} we denote
\begin{equation}
\label{def:F}
F(x):=f(x)-x, \quad x>0
\end{equation}
and remark that $F$ is positive on $(0, K)$ and negative on $(K, \infty)$.

\begin{assumption} 
\label{as:g34}
There exists $\lambda \in (0,1)$ such that for any $x>0$ either
\begin{equation}
\label{cond:g3}
|f(x)-K|\le \lambda |x-K|
\end{equation}
or
\begin{equation}
\label{cond:g4}
(f(x)-K)(x-K)>0.
\end{equation}
\end{assumption}

\begin{remark}
\label{rem:nonexpansion}
Note that Assumptions \ref{as:g12}-\ref{as:g34}  imply that, for $x>0$ and $x \neq K$,
\[
|f(x)-K| < |x-K|.
\]
Indeed, if $x\in (0, K)$ and $f(x)\in (0, K)$, then
\[
|f(x)-K|=K-f(x) <  K-x=|x-K|.
\]
If $x > K$ and $f(x) > K$, then
\[
|f(x)-K| = f(x)-K < x-K = |x-K|,
\]
while for $(f(x)-K)(x-K)<0$, 
\[
|f(x)-K|\le \lambda |K-x|<|K-x|.
\]
\end{remark}

\begin{remark}
\label{rem:lambda}
Note that inequality \eqref{cond:g3} implies that, for $x\ge K$
\[
-\lambda x+K(1+\lambda)\le  f(x)\le \lambda x+K(1-\lambda),
\]
while for $x \le K$
\[
\lambda x+K(1-\lambda)\le  f(x) \le -\lambda x+K(1+\lambda).
\]
\end{remark}
We also will be using the following additional assumptions on $f$ and perturbations $\gamma$. 

\begin{assumption} 
\label{as:fxinf} 
Eventually the difference $x-f(x)$ exceeds a positive constant:
\begin{equation}
\label{cond:gamma2}
\liminf_{x \to \infty} (x-f(x))>0.
\end{equation}
\end{assumption} 

\begin{assumption} 
\label{as:gammasum1}
The perturbation sequence tends to zero:
\begin{equation}
\label{cond:gamma1}
\lim_{i \to \infty} \gamma_i = 0.
\end{equation}
\end{assumption}

Let  $(\Omega, {\mathcal{F}},  (\mathcal{F}_n)_{n \in \mathbb{N}}, {\mathbb{P}})$ be  a complete filtered probability space.  Let $\xi:=(\xi_n)_{n\in\mathbb{N}}$  be a sequence of independent random variables.
The filtration $(\mathcal{F}_n)_{n \in \mathbb{N}}$ is supposed to be naturally generated by 
the sequence $(\xi_n)_{n\in\mathbb{N}}$, namely 
$\mathcal{F}_{n} = \sigma \left\{\xi_{1},  \dots, \xi_{n}\right\}$.

In the present paper we assume that the stochastic perturbation $\xi$ in equation \eqref{eq:stochintr} satisfies the following assumption.

\begin{assumption}
\label{as:chibound}
$\xi=(\xi_n)_{n\in \mathbb N}$ is a sequence of independent   continuous random variables with the density functions  $\phi_n(x)$  such that 
\[
\phi_n(x)>0, \quad x\in (-1, 1), \quad \phi_n(x)\equiv 0, \quad x\notin [-1, 1],  \quad n\in \mathbb N.
\]
\end{assumption}

We use the standard abbreviation ``a.s." for the wordings ``almost sure" or ``almost surely" with respect
to the fixed probability measure $\mathbb P$  throughout the text. A detailed discussion of stochastic concepts and
notation can be found in, for example, 
\cite{Shiryaev96}.

For further calculations it is convenient to note that, for any $x_n$, $x_{n+1}$, satisfying either equation \eqref{eq:ggammaintr} or \eqref{eq:stochintr}, the inequalities 
\begin{equation}
\label{ineq:det}
x_{n+1}\ge f(x_n)+\gamma_{n+1}, \quad  n=0,1, \dots,
\end{equation}
\begin{equation}
\label{ineq:stoch}
x_{n+1}\ge f(x_n)+\sigma_n\xi_{n+1}, \quad n=0,1, \dots,
\end{equation}
are valid, respectively.


\section{Deterministically perturbed difference equation}
\label{sec:det}

Consider the deterministically perturbed difference equation 
\begin{equation}
\label{eq:ggamma}
x_{n+1}= \max\left\{ f(x_n)+\gamma_{n+1}, 0 \right\}, \quad n=0,1, \dots,
\end{equation}
where for $f$ Assumption~\ref{as:g12} holds, and $\gamma_n$ is a deterministic  perturbation, satisfying \eqref{cond:gamma1}.  


\subsection{Auxiliary lemmata and main deterministic theorem}
\label{subsec:maindet}

Let us prove that under \eqref{cond:gamma1}, for any small enough $\varepsilon_0$, 
there exists $n_1 \in \NN$ such that $x_n \in [\varepsilon_0,2K-\varepsilon_0]$ for some $n>n_1$ implies $x_j \in [\varepsilon_0,2K-\varepsilon_0]$
for any $j \geq n$.

\begin{lemma}
\label{lem:add1}
Suppose that Assumptions~\ref{as:g12}, \ref{as:g34} and \ref{as:gammasum1} hold. Let 
$x$ be a solution of \eqref{eq:ggamma} with an arbitrary initial value $x_0>0$.
Then for any $\varepsilon_0>0$ satisfying
\begin{equation} 
\label{add1}
\varepsilon_0 < \min\left\{ \lambda K, \,\,  \frac{(1-\lambda)K}{2} \right\}
\end{equation}
there exists $n_1 \in \NN$ such that 
if $x_n \in [\varepsilon_0, 2K-\varepsilon_0]$ for some $n\geq n_1$, 
then  $x_{n+k} \in [\varepsilon_0, 2K-\varepsilon_0]$ for any $k \in \NN$.
\end{lemma}

\begin{proof}
Let $\varepsilon_0>0$ satisfy \eqref{add1}. Define $\delta_0>0$ and $n_1=n_1(\delta_0)$ such that
\begin{equation}
\label{add2}
\delta_0< \min\left\{ \min_{x \in [\varepsilon_0, 2\varepsilon_0]} (f(x)-x), \quad \lambda \varepsilon_0, \quad 
\min_{x \in [K+\varepsilon_0,2K-\varepsilon_0]} (x-f(x)) \right\},
\end{equation}
\[
|\gamma_n|\le \delta_0, \quad \text{for} \quad n\ge n_1.
\]
Then for $x \in [\varepsilon_0,2\varepsilon_0]$, $n \ge n_1$,  we have 
\[
f(x)+\gamma_n>f(x)-x+x-\delta_0>\delta_0+x-\delta_0 = x \geq \varepsilon_0,
\]
while  for $x \in [2\varepsilon_0, K]$, $n \ge n_1$,  we have 
\[
f(x)+\gamma_n>2\varepsilon_0-\delta_0>\varepsilon_0.
\]
Also, for  $x \in [\varepsilon_0, K]$, $n \ge n_1$,  we have, by \eqref{add1}, \eqref{add2} and Remark~\ref{rem:lambda}, 
\[
f(x)+\gamma_n<-\lambda x+K(1+\lambda)+\delta_0<-\lambda \varepsilon_0+K(1+\lambda)+\delta_0<
K(1+\lambda)<2K-\varepsilon_0.
\]
If $x\in [K, K+\varepsilon_0]$, by $\varepsilon_0 < \lambda K$,
\[
f(x)+\gamma_n<K+\varepsilon_0+\delta_0<K+2\varepsilon_0<K+(1-\lambda)K=2K-\lambda K<2K-\varepsilon_0.
\]
For $x\in [K+\varepsilon_0, 2K-\varepsilon_0]$, we have
\[
f(x)+\gamma_n<f(x)-x+x+\delta_0<-\delta_0+x+\delta_0 = x <2K-\varepsilon_0,
\]
while for $x\in [K, 2K-\varepsilon_0]$, by Remark~\ref{rem:lambda},
\[
f(x)+\gamma_n>-\lambda x+K(1+\lambda)-\delta_0>-2\lambda K+\lambda\varepsilon_0+K(1+\lambda)-\delta_0>K(1-\lambda)>\varepsilon_0.
\]
Thus, $f(x_n) + \gamma_{n+1} \in [\varepsilon_0, 2K-\varepsilon_0]$ as long as
$x_n \in [\varepsilon_0, 2K-\varepsilon_0]$ and $n \ge n_1$, which concludes the proof.
\end{proof}

\begin{lemma}
\label{lem:add2}
Suppose that Assumptions~\ref{as:g12}, \ref{as:g34}, \ref{as:fxinf} and \ref{as:gammasum1} hold, and $\varepsilon_0$ satisfies \eqref{add1}.  
For any solution $x$ of \eqref{eq:ggamma} with a positive  initial value
there exists $n_2  \in {\mathbb N}$ such that $x_n \le 2K-\varepsilon_0$ for $n \geq n_2$.
\end{lemma}
\begin{proof}
By condition \eqref{cond:gamma2} of Assumption~\ref{as:fxinf}, 
\begin{equation}
\label{eq:ast2}
\sigma:=\inf_{x>2K-\varepsilon_0} (x-f(x))>0. 
\end{equation}
Let us choose $\delta_0$ as in \eqref{add2} and $n_1 \in {\mathbb N}$ such that
$$ |\gamma_n| < \min\left\{ \frac{\sigma}{2}, \delta_0 \right\}, \quad n \geq n_1. $$
Then, as long as $x>2K-\varepsilon_0$, we have
$$
f(x)+ \gamma_{n+1}=f(x)-x+x+ \gamma_{n+1}\le -\sigma+x+\frac{\sigma}{2} < x - \frac{\sigma}{2},
$$
which implies the existence of $n_2 \geq n_1$ such that $x_{n_2} \leq 2K-\varepsilon_0$.

Let us prove that $x_n\leq 2K-\varepsilon_0$ for any $n \geq n_2$.  Since all the assumptions of Lemma~\ref{lem:add1} 
hold  and $n_2 \geq n_1$, the relation  $x_{n_2} \in [\varepsilon_0,2K-\varepsilon_0]$ implies 
$x_n \in [\varepsilon_0,2K-\varepsilon_0]$ for all  $n \geq n_2$.  For $x_{n_2} < \varepsilon_0$ we have, by 
Remark~\ref{rem:lambda} and \eqref{add1},
\[
x_{n_2+1}=f(x_{n_2}) +\gamma_{n_2+1}\le -\lambda x_{n_2} + K(1+\lambda)+\delta_0 \le K(1+\lambda) +\varepsilon_0< 2K-\varepsilon_0,
\]
Thus, in all cases, $x_n < 2K-\varepsilon_0$ for $n \geq n_2$
implies $x_j < 2K-\varepsilon_0$ for any $j \geq n$, which concludes the 
proof. 
\end{proof}

\begin{lemma}
\label{lem:add3}
Suppose that Assumptions~\ref{as:g12}, \ref{as:g34}, \ref{as:fxinf} and \ref{as:gammasum1} hold.
Let $x$ be a solution of \eqref{eq:ggamma} with an arbitrary initial value $x_0 > 0$,
and $\sigma_1>0$ be such that $\limsup\limits_{n \to \infty} x_n \geq \sigma_1$. Then there is an $\varepsilon_0>0$ satisfying 
\eqref{add1} 
and $n_0 \in \NN$ such that $x_n \in  [\varepsilon_0,2K-\varepsilon_0]$ for any $n \geq n_0$.
\end{lemma}
\begin{proof}
Take $\varepsilon_1 \leq \sigma_1/2$ satisfying \eqref{add1} and apply   Lemma~\ref{lem:add2}. Then, there exists $n_1\in \mathbf N$ such that $x_n \leq 2K-\varepsilon_1/2$
for all $n\ge n_1$. As $\limsup\limits_{n \to \infty} x_n \geq \sigma_1$, we can choose  $n_0\ge n_1$ such that 
$x_{n_0}>\sigma_1/2$. Define $\varepsilon_0:=\varepsilon_1/2$, which also 
satisfies \eqref{add1}.  Then  $x_{n_0} \in [\varepsilon_0, 2K-\varepsilon_0]$, which, by Lemma~\ref{lem:add1},  
implies $x_{n_0+k} \in  [\varepsilon_0,2K-\varepsilon_0]$ for any $k \in \NN$.
\end{proof}

\begin{corollary}
\label{cor1}
Suppose that Assumptions~\ref{as:g12}, \ref{as:g34} and \ref{as:gammasum1}  hold, and $\varepsilon_0$ satisfies \eqref{add1}.  
Then any solution $x$ of equation \eqref{eq:ggamma} either tends to zero or is in $[\varepsilon_0,2K-\varepsilon_0]$,
starting with some $n_0$.
\end{corollary}

\begin{theorem}
\label{thm:detmain} 
Suppose that Assumptions~\ref{as:g12}, \ref{as:g34}, \ref{as:fxinf} and \ref{as:gammasum1} hold. 
Then any  solution $x$ of \eqref{eq:ggamma}  converges either to $K$ or to zero. 
\end{theorem}
\begin{proof}
By Lemma~\ref{lem:add3}, it is sufficient to consider sequences $(x_n)_{n\in \mathbb N}$ which are in 
$[\varepsilon_0,2K-\varepsilon_0]$, for some $\varepsilon_0$ satisfying  \eqref{add1} and $n \ge n_0$.
Let us fix $\delta< \varepsilon_0$ and prove that $x_n \in (K-\delta,K+\delta)$ for $n$ large enough. This will  imply that any  sequence not converging to zero will converge to $K$.

Let $x_n \in [\varepsilon_0,K-\delta] \cup [K+\delta, 2K-\varepsilon_0]$.
By Assumption~\ref{as:g12}, there exist 
\begin{equation}
\label{def_sigma}
\sigma := \min\left\{ \min_{x\in [\varepsilon_0, K-\delta]} (f(x)-x),  \min_{x\in [K+\delta, 2K-\varepsilon_0]} (x-f(x)) \right\} 
>0 
\end{equation}
and $n_1 \geq n_0$ such that 
\begin{equation}
\label{est:gamma}
|\gamma_n|< \min \{ \delta, \sigma \} \frac{1-\lambda}{2} \quad \mbox{for} \quad n \geq n_1.
\end{equation}
Then, for $x_n \in [\varepsilon_0, K-\delta]$, $n>n_1$
from \eqref{est:gamma} and the definition of $\sigma$ in \eqref{def_sigma}, we have
$$
x_{n+1} = f(x_n)+ \gamma_{n+1} \geq f(x_n)-\sigma/2 \geq x_n+\sigma-\sigma/2=x_n+\sigma/2,
$$
and similarly $x_{n+1} \leq x_n-\sigma/2$ for $x_n \in [K+\delta, 2K-\varepsilon_0]$, $n>n_1$. Thus, if $(x_{n+1}-K)(x_n-K)>0$,
we have
$$|x_{n+1}-K| \leq |x_n-K|-\sigma/2.
$$  
Now, let $x_n<K$ and $x_{n+1}>K$. 
Then either $x_{n+1}\in (K ,K+\delta)$ or $x_{n+1}>K+\delta$. 
In the latter case, by \eqref{est:gamma},  we also have $(f(x_n)-K)(x_n-K)<0$, since 
\[
f(x_n)-K\ge x_{n+1}-K-|\gamma_{n+1}|\ge \delta\left[ 1-\frac{1-\lambda}{2}\right]>0.
\]
Applying  Assumption~\ref{as:g34} and recalling $\displaystyle |x_n-K| \geq \delta$, we get
\begin{eqnarray*}
|x_{n+1}-K| & \leq & |f(x_n)-K| + |\gamma_{n+1}| < \lambda |x_n-K|+ (1-\lambda)\frac{\delta}{2}
\\ & \leq & |x_n-K|-\frac{1-\lambda}{2}|x_n-K|  = \frac{1+\lambda}{2} |x_n-K|,
\end{eqnarray*}
and also $|x_{n+1}-K| \leq |x_n-K|-(1-\lambda) \delta/2$.  The case $x_{n+1}>K$ is treated similarly.

Next,
$$|x_{n+1}-K| \leq |x_n-K| -\min\{ \sigma/2, (1-\lambda) \delta/2 \},$$
as long as $|x_n-K|\geq \delta$, thus for any $n_2 \in {\mathbb N}$ there is $n>n_2$ such that 
$x_n \in (K-\delta,K+\delta)$. 

By Assumption~\ref{as:g12}, there exist $\delta_1 \in (0,\delta)$, $\delta_2\in (0,\delta)$ such that
\begin{equation}
\label{delta_i}
\min_{x\in[K-\delta,K]}f(x) \geq K-\delta_1 > K - \delta, \quad \max_{x\in[K,K+\delta]} f(x)\le K-\delta_2 < K + \delta.
\end{equation}
Assume that $n_2 \geq n_1$ is such that in addition to 
$|\gamma_n|< \min \{ \delta, \sigma \} \frac{1-\lambda}{2}$, we have $|\gamma_n|< \min \{ \delta-\delta_1, \delta-\delta_2\}$.

Let $x_n \in (K-\delta,K+\delta)$ for $n \geq n_2$.  It remains to prove that $x_{n+1} \in (K-\delta,K+\delta)$. In fact,
if $x_n \in (K-\delta, K)$ and $f(x_n) < K$  then 
$$x_{n+1} = f(x_n)+ \gamma_{n+1} \leq K+ (1-\lambda) \frac{\delta}{2} < K+ \frac{\delta}{2}< K+\delta.$$ 
Also, as $|\gamma_n|< \delta-\delta_1$,
$$x_{n+1} = f(x_n)+ \gamma_{n+1} \geq K-\delta_1+ \gamma_{n+1} > K-\delta_1 - (\delta -\delta_1) =K-\delta.$$
Similarly, if $x_n \in (K, K+\delta)$ and $f(x_n) > K$ we have
$$x_{n+1} = f(x_n)+ \gamma_{n+1} \geq K- (1-\lambda) \frac{\delta}{2} > K - \frac{\delta}{2}> K-\delta.$$
Since $|\gamma_n|< \delta-\delta_2$, we have
$$x_{n+1} = f(x_n)+ \gamma_{n+1} \leq K+\delta_2+ \gamma_{n+1} < K+\delta_2 + (\delta -\delta_2) =K+ \delta.$$

If $(x_{n+1}-K)(x_n-K)<0$ we act as above and arrive at
\[
|x_{n+1}-K| \leq \lambda |x_{n}-K| +|\gamma_{n+1}|\le \lambda \delta + (1-\lambda) \frac{\delta}{2}
<\delta.
\]
So,  in all cases $x_{n+1}\in (K-\delta,K+\delta)$.   

Finally, if there is no $\varepsilon$ such that $\limsup_{k\to \infty}x_k>\varepsilon$, then
$\lim\limits_{n\to\infty} x_n=0$, which completes the proof.
\end{proof}


\subsection{Cases when solutions cannot tend to zero} 
\label{subsec:detnottozero}

In this section we derive conditions on $f$, as well as the connection between $f$ and perturbation $\gamma$,  that guarantee  
\begin{equation}
\label{lim:xnotzero}
\lim_{n\to \infty} x_n\neq 0.
\end{equation}
Since we are interested only in \eqref{lim:xnotzero}, we actually do not need all parts of Assumption \ref{as:g12}. It is enough to suppose only that 
\begin{equation}
\label{cond:xnotzero}
\text{$f:\RR^+\to \RR^+$ is continuous, $f(0)=0$, $f(K)=K$, $f(x)>x$ for all $x\in (0, K)$.}
\end{equation}
We present  examples of equations of type \eqref{eq:ggamma} with solutions  $x_n$ satisfying $\lim\limits_{n\to \infty}x_n=0$.  We also discuss the case when  negative perturbation terms are small, and derive the minimum initial value which guarantees \eqref{lim:xnotzero}.

The following proposition shows that $\lim\limits_{n\to \infty}x_n=0$  is impossible when $\gamma_n$ 
alternatively changes sign  and some connection between $f$ and $\gamma_n$ is imposed, which generalizes 
the property of  monotonicity of the sequence $(\gamma_n)_{n\in \mathbf N}$.

\begin{proposition}
\label{prop:pmgammanew}
Let conditions \eqref{cond:gamma1} and \eqref {cond:xnotzero}  hold.  Assume  that $\gamma_n=(-1)^n\beta_n$, $n\in \mathbb N$,  where $\beta_n>0$ for all $n\in \mathbb N$, 
and 
\begin{equation}
\label{eq:fdernew}
f(x+a)-f(x)>a, \quad \text{for all} \quad x, a\in (0, \bar\delta),
\end{equation}
and 
\begin{equation}
\label{eq:fgammabeta}
f\left(\beta_{2k}  \right)>\beta_{2k+1}, \quad \text{for all big enough} \quad k\in \mathbb N.
\end{equation}
Then $\lim\limits_{n\to \infty} x_n\neq 0$ for any solution $x_n$ of \eqref{eq:ggamma}.
\end{proposition}
\begin{proof}
Note that condition \eqref{eq:fdernew} implies that $f$ is increasing on $(0, \bar \delta)$.

Assuming that $\lim_{n\to \infty} x_n=0$, for $\delta\in (0, \bar \delta)$ 
we can find $N_\delta\in \mathbb N$ such that 
\[
x_n<\delta, \quad n\ge N_\delta.
\]

Without loss of generality we can suppose that $N_\delta$ is even, so $\gamma_{N_\delta}=\beta_{N_\delta}$. Then,
\begin{eqnarray}
x_{N_\delta}&=&f(x_{N_\delta-1})+\gamma_{N_\delta}\ge \gamma_{N_\delta}=\beta_{N_\delta}.\nonumber\\
x_{N_\delta+1}&=&f(x_{N_\delta})+\gamma_{N_\delta+1}\ge f(\beta_{N_\delta})-\beta_{N_\delta+1}>0.\label{est:lambdaN0}
\end{eqnarray}
Denote $\lambda_{N_\delta}:=f(\beta_{N_\delta})-\beta_{N_\delta+1}$. Then
\begin{eqnarray}
x_{N_\delta+2}&=&f(x_{N_\delta+1})+\gamma_{N_\delta+2}\ge 
x_{N_\delta+1}+\beta_{N_\delta+2}\ge f(\beta_{N_\delta})-\beta_{N_\delta+1}+\beta_{N_\delta+2}\nonumber\\
&=&\lambda_{N_\delta}+\beta_{N_\delta+2}.\nonumber\\
x_{N_\delta+3}&=&f(x_{N_\delta+2})+\gamma_{N_\delta+3}\ge f(\lambda_{N_\delta}+\beta_{N_\delta+2})-\beta_{N_\delta+3}\label{est:lambdaN}\\
&=&f(\lambda_{N_\delta}+\beta_{N_\delta+2})-f(\beta_{N_\delta+2})+f(\beta_{N_\delta+2})-\beta_{N_\delta+3}\nonumber\\
&\ge &\lambda_{N_\delta}
+f(\beta_{N_\delta+2})-\beta_{N_\delta+3}
\ge\lambda_{N_\delta} .\nonumber
\end{eqnarray}
Applying induction we can prove that for each $k\in \mathbb N$,
$$
x_{N_\delta+2k}\ge \lambda_{N_\delta}+\beta_{N_\delta+2k}, \quad
x_{N_\delta+2k+1}\ge \lambda_{N_\delta},
$$
which contradicts to the assumption that $\lim_{n\to \infty} x_n=0$.   
\end{proof}

\begin{corollary}
\label{cor:altsign}
If in Proposition \ref{prop:pmgammanew} we omit \eqref{eq:fdernew} and instead of condition \eqref{eq:fgammabeta} assume 
that $(\beta_n)$ is a nonincreasing sequence, then $\lim\limits_{n\to \infty} x_n\neq 0$ for any solution $x_n$ of \eqref{eq:ggamma}.  
\end{corollary}
\begin{proof}
Indeed, when  $(\beta_n)$ is nonincreasing sequence, we get 
$x_{N_\delta+1}>\beta_{N_\delta}-\beta_{N_\delta+1}$ instead of \eqref{est:lambdaN0}, and inductively, 
for each $k\in \mathbb N$,
\begin{eqnarray*}
&&
x_{N_\delta+2k+1}\ge \beta_{N_\delta}-\beta_{N_\delta+1}+\sum_{i=2}^{2k}[\beta_{N_\delta+i}-\beta_{N_\delta+i+1}]>
\beta_{N_\delta}-\beta_{N_\delta+1}>0,
\\ && 
x_{N_\delta+2(k+1)}\ge\beta_{N_\delta}-\beta_{N_\delta+1}+\beta_{N_\delta+2k+2}>\beta_{N_\delta}-
\beta_{N_\delta+1}.
\end{eqnarray*}
\end{proof}

In order to generalize the approach of Proposition \ref{prop:pmgammanew}, we combine the consecutive perturbations $\gamma_n$ with 
the same sign into groups. 
Without loss of generality assume that $\gamma_1>0$. Also assume that there are infinitely many negative and infinitely many positive  $\gamma_n$. Let
\[
n_0=\inf\{i>1: \gamma_i<0\}-1,  \quad n_1=\inf\{i>n_0: \gamma_i>0\}-1, 
\] 
and similarly, for $k=1, 2, \dots$
\[
n_{2k}=\inf\{i> n_{2k-1}: \gamma_i<0\}-1,  \quad n_{2k+1}=\inf\{i>n_{2k}: \gamma_i>0\}-1.
\] 
All the sets above are non-empty, so we can define
\begin{equation}
\label{def:betan}
\beta_{2k}:=\sum_{i=n_{2k-1}+1}^{n_{2k}}\gamma_i>0, \quad \beta_{2k+1}:=-\sum_{i=n_{2k}+1}^{n_{2k+1}}\gamma_i>0.
\end{equation}
Note that all $\gamma_i$ have the same sign in each of the above sums.

\begin{proposition}
\label{prop:altgroup}
Let conditions \eqref{cond:gamma1}, \eqref {cond:xnotzero},  \eqref{eq:fdernew} and \eqref{eq:fgammabeta}
hold, and there be infinitely many negative and infinitely many positive  $\gamma_n$. Let $\beta_n$ be defined as in \eqref{def:betan} and $\lim_{n\to \infty} \beta_n=0$. 
Then $\lim\limits_{n\to \infty} x_n\neq 0$ for any solution $x_n$ of \eqref{eq:ggamma}.
\end{proposition}  
\begin{proof}
The proof is similar to the proof of Proposition \ref{prop:pmgammanew}.  We
 start the estimation from $n_{2k-1}$ such that $n_{2k-1}>N_\delta$ and \eqref{eq:fgammabeta} holds.  Applying  conditions \eqref {cond:xnotzero} and monotonicity of $f$, we get
\begin{eqnarray*}
x_{n_{2k-1}+1}&=&f(x_{n_{2k-1}})+\gamma_{n_{2k-1}+1}\ge \gamma_{n_{2k-1}+1}>0,\\
x_{n_{2k}}&\ge &\sum_{i=1}^{n_{2k}-n_{2k-1}}\gamma_{n_{2k-1}+i}=\beta_{2k},\\
x_{n_{2k+1}}&\ge &f(\beta_{2k})+ \sum_{i=1}^{n_{2k+1}-n_{2k}}\gamma_{n_{2k}+i}= f(\beta_{2k})-\beta_{2k+1}:=\lambda_k,\\
\text{and} &&\\
x_{n_{2k+2}}&\ge& \lambda_k+\beta_{2k+2}, \quad x_{n_{2k+3}}\ge f\left(\lambda_k+\beta_{2k+2}\right)-\beta_{2k+3}.
\end{eqnarray*}
Acting as in  \eqref{est:lambdaN} we arrive at
\begin{equation*}
\begin{split}
x_{n_{2k+3}}\ge f\left(\lambda_k+\beta_{2k+2}\right)- f\left(\beta_{2k+2}\right)+f\left(\beta_{2k+2}\right)-\beta_{2k+3}
\ge \lambda_k + f\left(\beta_{2k+2}\right)-\beta_{2k+3}\ge \lambda_k.
\end{split}
\end{equation*}
Applying induction we show $x_{n_{2k+s}} \ge \lambda_k$, for each  $s\in \mathbb N$.
\end{proof}

\begin{remark}
\label{rem:fgamma}
Let  \eqref {cond:xnotzero} be fulfilled.
Relation \eqref{eq:fdernew} holds in particular, when  $f$ is differentiable on $(0, K)$ and the derivative $f'$ is greater than 1 in some right neighborhood of $0$ (however $f'(0)$ can be equal to 1). 

Note that \eqref{eq:fgammabeta} is a less restrictive condition than monotonicity of $\beta_n$.  
Relation \eqref{eq:fgammabeta} holds in the following cases:
\begin{enumerate}
\item [(i)] sequence $(\beta_n)_{n\in \mathbb N}$ is decreasing;
\item [(ii)] $f(x)\ge (1+\mu)x$ for $x\in (0, \delta)$ and  $(1+\mu)\beta_{2k}\ge \beta_{2k+1}$ for some $\mu>0$, 
$\delta \in (0, K)$ and all $k\in \mathbb N$;
\item [(iii)] $f(x)\ge x+x^\nu$ for $x\in (0, \delta)$ and 
$\beta_{2k}+\left( \beta_{2k}\right)^\nu\ge \beta_{2k+1}$ for some $\delta\in (0, K)$, $\nu>1$ and all $k\in \mathbb N$.
\end{enumerate}
Note that in the cases (ii) and (iii) the sum of consecutive negative perturbations can be greater than the sum of positive ones. 
For example, in (ii) it can be  $(1+\mu)\beta_{2k}\ge \beta_{2k+1}> \beta_{2k}$.
Also, the bigger the derivative $f'(0)$ is, the less restrictions we need to impose on quasi-monotonicity of $\beta_n$.
\end{remark}

The following examples show that $\lim\limits_{n\to \infty}x_n=0$ is possible for equations with quickly decaying perturbations which 
either remain negative after some moment, or do not satisfy 
condition~\eqref{eq:fgammabeta}.
In Example \ref{ex:exlim0} the function $f$ grows quickly in the right neiborhhood of zero (actually $f'(0)=\infty$),  $\gamma_n<0$ and $\lim\limits_{n\to \infty} x_n=0$.

\begin{example}
\label{ex:exlim0}
Let $K=1$, $f(x)=\sqrt{x}$, $x \geq 0$ and 
$\displaystyle
\gamma_{n+1}=-\frac 1{n^2}+\frac 1{(n+1)^4}<0$.
Then Assumption \ref{as:g12} and  condition \eqref{cond:g4} hold. 
However, $x_n=\frac{1}{n^4}$ is a solution of the equation
\begin{equation}
\label{ex1eq1}
x_{n+1}=\max \left\{ \sqrt{x_n}+\gamma_{n+1}, 0 \right\}, \quad x_1=1.
\end{equation}
In addition,
$\displaystyle
\sum_{n=1}^\infty|\gamma_n|<\infty, \quad f'(0)=\infty,
$
~$
f(x)>2x, \quad x\in (0, 0.25)$,
$\displaystyle |\gamma_{n}|=\frac 1{n^2}-\frac 1{(n+1)^4}$ decreases and $\lim\limits_{n\to \infty}\gamma_n=0$.
Let us note that all solutions of \eqref{ex1eq1} with $x_1 \in (0,1]$ either tend to zero or are identically equal to zero, starting 
with a certain $n \in \NN$.
\end{example}

In Example  \ref{ex:exlim1} the function $f$ is the same as in Example  \ref{ex:exlim0}, $\gamma_n$ is an alternating sequence, 
$(|\gamma_n|)_{n\in \mathbf N}\in {\bf \ell}_2$ and $\lim\limits_{n\to \infty} x_n=0$.

\begin{example}
\label{ex:exlim1}
Let $f(x)=\sqrt{x}$, $x\geq 0$, $\varepsilon>0$, $x_2=\frac{1}{16}$ and 
$\displaystyle \gamma_{2(n+1)}=-\frac {\sqrt{1+\varepsilon}}{2n}+\frac 1{[2(n+1)]^4}<0$,
while
$\displaystyle 
\gamma_{2n+1}= \frac {\varepsilon}{(2n)^2}>0$.
Then
\[
x_{2n}=\frac 1{(2n)^4}, \quad x_{2n+1}=\frac {1+\varepsilon}{(2n)^2}.
\]
Indeed, $\displaystyle x_{2n+1}=\frac 1{(2n)^2}+\frac {\varepsilon}{(2n)^2}=\frac {1+\varepsilon}{(2n)^2}$ and 
\[
x_{2(n+1)}=\frac {\sqrt{1+\varepsilon}}{2n}-\frac {\sqrt{1+\varepsilon}}{2n}+\frac 1{[2(n+1)]^4}
= \frac 1{[2(n+1)]^4}.
\]
\end{example}

\begin{remark}
Function $f$ and perturbations $\gamma$ in Examples~\ref{ex:exlim0}-\ref{ex:exlim1}
do not satisfy \eqref{eq:fgammabeta}. 
In Example~\ref{ex:exlim0} there are no positive perturbations, in Example \ref{ex:exlim1} we have, for $n$ big enough,
\[
f(\gamma_{2n+1})=f(\beta_{2n})=\frac {\sqrt{\varepsilon}}{2n}<\frac {\sqrt{\varepsilon+1}}{2n}-\frac {1}{[2(n+1)]^2}=\beta_{2n+1}=\gamma_{2(n+1)}.
\]
\end{remark}

Even though Examples \ref{ex:exlim0}-\ref{ex:exlim1} demonstrate the possibility for solution $x_n$ to tend to zero, in Proposition 
\ref{prop:invalpar}  we show that in some cases there exists $b>0$ such that \eqref{lim:xnotzero} still holds when $x_0>b$.  
We neither require monotonicity of perturbations $\gamma_n$  nor any conditions of type~\eqref{eq:fgammabeta}.  
However we require $f$ to be increasing in some right neighborhood of zero and the negative perturbations to be small enough. 

More precisely, we assume that
\begin{equation}
\label{cond:flambda}
\text{$f$ increases on $[0, c]$ for some $c>0$ and $\min _{x\in (c, \infty)}f(x)\ge \lambda$, for some $\lambda>0$.}
\end{equation}
For bounded real sequence $(\gamma_n)_{n\in \mathbb N}$  we define $\tilde \gamma:=\sup_{n\in \mathbb N}| \gamma_n^-|$ and 
\begin{equation}
\label{def:btildeg}
b=b(\tilde \gamma):=\inf\{x>0: f(x)-x>\tilde \gamma   \}.
\end{equation}
Note that for small enough $\tilde \gamma$ the set in the right-hand-side of \eqref{def:btildeg} is non-empty.  Asssume that $\tilde \gamma$ is so small that 
\begin{equation}
\label{cond:tildeg}
\tilde \gamma+b(\tilde \gamma)<\lambda.
\end{equation}

\begin{proposition}
\label{prop:invalpar}
Let Assumption  \ref{as:g12}  and conditions \eqref{cond:gamma1},  \eqref{cond:flambda} and \eqref{cond:tildeg} hold.  
Then any solution $x_n$ of equation \eqref{eq:ggamma} with the initial value $x_0\in \left(b(\tilde \gamma), \infty\right)$  
satisfies $x_n> b(\tilde \gamma)$.
\end{proposition}
\begin{proof}
In fact, for each $n\in \mathbb N$,
$
\gamma_n=\gamma_n^++\gamma_n^->\min\{0, -\gamma_n^-\}\ge -\tilde \gamma$.

If  $x_0\in (b, c)$ then  
$
x_1=f(x_0)+\gamma_1\ge f(b)+\gamma_1=f(b)-b+b+\gamma_1\ge \tilde \gamma +b-\tilde \gamma=b.
$
If $x_0\ge c$, then by \eqref{cond:tildeg} we have
$
x_1=f(x_0)+\gamma_1\ge \lambda-\tilde \gamma >
b.
$
Similarly, if $x_{n}\in (b, c)$,
$
x_{n+1}=f(x_n)+\gamma_{n+1}\ge \tilde \gamma +b-\tilde \gamma=b$, and
if $x_n\ge c$,
$
x_{n+1}=f(x_n)+\gamma_{n+1}\ge \lambda-\tilde \gamma>
b$ by \eqref{cond:tildeg}.
Applying induction, we conclude that the solution is persistent, where $b$ is the lower 
bound of the solution.
\end{proof}

\section{Some preliminary results on stochastic sequences}
\label {sec:prel}

In this section we present several results about the stochastic sequences, which will be useful in finding the probability  
that a solution of stochastic difference equation \eqref{eq:stochintr} converges to the equilibrium point $K$.

The following lemma states that when  $\xi_n$ are independent identically distributed random variables, for any subinterval in their support we can find any number of consecutive $\xi_n$ taking values in this subinterval with probability 1.
\begin{lemma}
\label{lem:kpos}
If Assumption \ref{as:chibound} holds and $\xi_n$ are identically distributed, 
then for each $\varepsilon\in (0, 1)$, $J\in \mathbb N$ and $L\in \mathbb N$ there exists a.s. finite random number $N_J\in \mathbb 
N$, $N_J\ge L$ such that 
\begin{equation}
\label{cond:fkappa}
\mathbb P \left\{ \omega\in \Omega: \xi_n\in [1-\varepsilon, 1], \quad n = N_J+1, \dots,  N_J+J\right\}=1.
\end{equation}
\end{lemma}

\begin{proof}
Assumption \ref{as:chibound} implies that, for each $n\in \mathbb N$, 
\[
 \mathbb P\{\Omega_\varepsilon(n)\}=p_\varepsilon>0, \quad \text{where} \quad  \Omega _\varepsilon(n):=\{\omega\in \Omega:\xi_n(\omega)\in [1-\varepsilon, 1]\}.
\]
By independence of $\xi$, for each $n, J \in \mathbb N$, 
\[
\mathbb P \left\{ \omega\in \Omega: \xi_i\in [1-\varepsilon, 1], \,\,  i=n+1, \dots, n+J\right\}=p_\varepsilon^J.
\]
Therefore, for each $n\in\mathbb N$, the probability that among the random variables $\xi_{n+1}, \, \xi_{n+2}, \, \dots, \xi_{n+J}$ 
there is at least one which is not in $[1-\varepsilon, 1]$ is $1-p_\varepsilon^J$. 

Define
\begin{equation}
\label{def:Bi0}
B_i:=
\{ \omega\in \Omega: \xi_s\in [1-\varepsilon, 1], \quad s=(i-1)J+1,  \, (i-1)J+2, \dots,\,  iJ  \} 
\text{for} \quad i\ge i_0:=\frac L J+1.
\end{equation}
The conclusion of the lemma is valid if 
\[
\mathbb P \left\{\text{There exists $N_J>L$ such that $B_{N_J}$ occurs}\right\}=1.
\]
Since
\[
\mathbb P \left\{\text{At least one of  $B_{i}$ occurs, } \,\, i> i_0  \right\}=1-\mathbb P \left\{\text{All $\overline B_{i}$ occur, }  \,\,  i> i_0\right\}
\]
this is equivalent to  
\begin{equation}
\label{calc:prob0}
\mathbb P \left\{\text{All $\overline B_{i}$ occur, } \,\,  i> i_0 \right\}=0.
\end{equation}
By independence of $B_i$ we have
\[
\mathbb P \left\{\text{All $\bar B_{i}$  occur}, i=i_0+1, i_0+2, \dots, i_0+n \right\}=\mathbb P \left\{\prod_{i=i_0+1}^{i_0+n} \bar B_{i} \right\}=
[1-p_\varepsilon^J]^n.
\]
However,
\[
\left\{\text{All $\bar B_{i}$ occur, } \,\, i>i_0 \right\}\subseteq \left\{\text{All $\bar B_{i}$ occur, $i=i_0+1, i_0+2, \dots, i_0+n$.} \right\},
\]
and then,
\[
\mathbb P \left\{\text{All $\bar B_{i}$ occur, } \,\, i\in \mathbb N \right\}\le  \mathbb P\left\{\text{All $\bar B_{i}$ occur, $i=1, 2, \dots, n$.} \right\}=[1-p_\varepsilon^J]^n.
\]
Since in above inequality, $n$ can be arbitrarily large, this implies \eqref {calc:prob0}. So, for some  random $i_1>i_0$  the event 
$B_{i_1}$ occurs with probability 1.  By \eqref{def:Bi0}  we have $(i_1-1)J>(i_0-1)J\ge L$.
So we have proved \eqref{cond:fkappa} for $N_J:=(i_1-1)J$.
\end{proof}

\begin{remark}
\label{rem:noniid}
We can relax the assumption that $\xi_n$ are identically distributed by assuming instead that there exists $\varepsilon\in (0, 1)$  and a number $p_\varepsilon>0$ such that, $\forall n\in \mathbb N$,
\begin{equation}
\label{eq:estprob}
\mathbb P \left\{ \omega\in \Omega: \xi_n\in [1-\varepsilon, 1]\right\}\ge p_\varepsilon.
\end{equation}
Repeating the proof of Lemma \ref{lem:kpos}, instead of equality, we get now inequalities 
\[
\mathbb P \left\{ \omega\in \Omega: \xi_i\in [1-\varepsilon,  1], \quad i=n, n+1, \dots, n+J.\right\}\ge p_\varepsilon^J,
\]
\[
\mathbb P \left\{\text{All $\bar B_{i}$  occur}, i=1, 2, \dots, n. \right\}=\mathbb P \left\{\prod_{i=1}^n \bar B_{i} \right\}\le 
[1-p_\varepsilon^J]^n.
\]
So,
\[
\mathbb P \left\{\text{All $\bar B_{i}$ occur, } \,\, i\in \mathbb N \right\}\le  \mathbb P\left\{\text{All $\bar B_{i}$ occur, $i=1, 2, \dots, n$.} \right\}\le [1-p_\varepsilon^J]^n.
\]
Note that, in particular,  \eqref{eq:estprob} holds  if, for each $t\in (0, 1)$,
\[
\inf_{n\in \mathbb N}\{\phi_n(t)\}\ge \tilde\phi(t)>0,
\]
where $\tilde \phi$ is continuous and $\phi_n$ is a density function for the random variable $\xi_n$.
\end{remark}

In Lemma~\ref{lem:lim0} we claim that if a random sequence converges to zero on a set with non-zero probability $p$ then, for each 
$\alpha\in (0, p)$, it  converges uniformly on the set with a smaller probability $p-\alpha$. 
The proof of the this result, as well as of Lemma~\ref{lem:infty1}, is deferred to Appendix.

\begin{lemma}
\label{lem:lim0}
Let $(x_n)_{n\in \mathbb N}$ be a random sequence, $\bar \Omega:=\{\omega\in \Omega: \lim_{n\to \infty} x_n(\omega)=0\}$ and $\mathbb P\{\bar \Omega\}=p>0$.

Then, for each $\delta, \alpha>0$,  there exists $\bar N=\bar N(\delta, \alpha)>0$ and $\Omega(\delta, \alpha)\subseteq \bar \Omega$, $\mathbb P\{\Omega(\delta, \alpha)\}\ge p-\alpha$, such that, for all $n\ge \bar N$, $\omega \in \Omega(\delta, \alpha)$,
\[
x_n(\omega)\in [0, \delta).
\]
\end{lemma}

Now we formulate a so called Two-Series Theorem (see e.g. \cite[Theorem 2, p. 386]{Shiryaev96}), which will be used later.
\begin{theorem}
\label{thm:2ser}
A sufficient condition for the convergence of the series $\sum \zeta_n$ of independent random variables $\zeta_n$, with probability 1, is that both series $\sum \mathbb E\zeta_n$ and $\sum \mathbb Var\zeta_n$ converges. If in addition, for some $c>0$, $\mathbb P[|\zeta_n|\le c]=1$, the condition is also necessary.
\end{theorem}

\begin{lemma}
\label{lem:infty1}
Let Assumption \ref{as:chibound} hold, $\mathbb E\xi_n=\mu_n$,  $\mathbb E\xi_n^2=1$, $\sigma=(\sigma_n)\notin {\bf \ell}_2$, 
$\mu^-=(\mu_n^-)\in {\bf \ell}_2$. Then, a.s.,
\begin{equation}
\label{eq:+infty}
\limsup_{n\to \infty} \sum_{k=1}^n \sigma_k \xi_{k+1}=\infty.
\end{equation}
\end{lemma}


\section{Stochastic equations}
\label{sec:stoch}

Consider the equation
\begin{equation}
\label{eq:main}
x_{n+1}=\max\{f(x_n)+\sigma_n\xi_{n+1}, \, 0\},  \quad x_0>0, \quad n=0, 1, \dots,
\end{equation}
where $f$ is continuous and satisfies Assumptions \ref{as:g12}, \ref{as:g34} and \ref{as:fxinf}, $\sigma_n>0$, and for the random sequence  $(\xi_n)_{n\in \mathbb N}$  Assumption \ref{as:chibound} holds.  

Note that under Assumption \ref{as:chibound}, when the support of all $\xi_n$ is $[-1, 1]$, the relation
\begin{equation}
\label{cond:noisetozero}
\lim_{n\to \infty}\sigma_n\xi_{n+1}=0, 
\end{equation}
holds on all $\Omega$ whenever   
\begin{equation}
\label{cond:sigmatozero}
\lim_{n\to \infty}\sigma_n=0 .
\end{equation}
Consider convergence of the solution $x$ of \eqref{eq:main} to the unique positive equilibrium point $K$ of $f$.  Theorem \ref{thm:detmain} implies that,  under Assumptions  \ref{as:g12}-\ref{as:fxinf}  and condition \eqref{cond:sigmatozero}, there are only two possibilities for the limiting behavior of solution:  it either tends to $K$, or tends to zero. Our main goal is to derive conditions eliminating the last possibility or at least to estimate its probability.
In other words, we want either  to show that 
\begin{equation}
\label{cond:prob0}
\mathbb P\{\omega\in \Omega: \lim_{n\to \infty} x_n(\omega)=0\}=0
\end{equation}
or estimate the probability $\mathbb P\{\omega\in \Omega: \lim\limits_{n\to \infty} x_n(\omega)=0\}$.

Since we are mostly interested in the behavior of $f$ in the right neighborhood zero, 
instead of using Assumptions \ref{as:g12}, \ref{as:g34} and \ref{as:fxinf} which deal with the global behavior of $f$, 
we only assume that condition \eqref{cond:xnotzero} is fulfilled, i.e.
$f:\RR^+\to \RR^+$ is continuous, $f(0)=0$, $f(K)=K$, $f(x)>x$ for all $x\in (0, K)$.

In Section \ref{subsec:notl2} we prove that when $\sigma\notin {\bf \ell}_2$, condition \eqref {cond:prob0} holds without any extra 
assumption on $f$.  The case $\sigma\in {\bf \ell}_2$ is discussed in Sections~\ref{subsec:consigmaF} and \ref{subsec:gencond}.

In Section \ref{subsec:consigmaF} we prove \eqref{cond:prob0} where there exists  a certain  connection between the noise intensity 
$\sigma_n$ and $F(x)=f(x)-x$. Our result holds in particular for the case when 
$\displaystyle
\liminf_{x \to 0^+} \frac {f(x)}{x}>1, 
$
which is quite common in population modelling.

In Section \ref{subsec:gencond} we suppose  that the distributions of $\xi_n$ are symmetrical. Assuming only \eqref{cond:xnotzero}, we prove that 
\[
\mathbb P\{\omega\in \Omega: \lim_{n\to \infty} x_n(\omega)=0\}\le \frac 12.
\]
In Section \ref{subsec:genstochthm} we summarize the obtained results on the equilibrium points of equation \eqref{eq:main}.


\subsection{Case $\sigma\notin {\bf \ell}_2$. }
\label{subsec:notl2}
In this section we do not apply any extra assumptions on $\lim\limits_{x\to 0^+}f(x)$.

\begin{lemma}
\label{lem:+infty}
Assume that $f$ satisfies \eqref{cond:xnotzero}  and  condition \eqref{eq:+infty} holds a.s.
Then \eqref{cond:prob0} holds for any solution $x_n$ to equation \eqref{eq:main} with the initial value $x_0>0$. 
\end{lemma}
\begin{proof}
Suppose the opposite: there exists $\Omega_p\subseteq \Omega$, $\mathbb P\{\Omega_p\}=p>0$ such that $\lim\limits_{n\to \infty} 
x_n(\omega)=0$  for $\omega\in \Omega_p$. Fix some $\delta \in (0, K).$ 
By Lemma \ref{lem:lim0} for any $\delta\in (0, K) $, there exists a nonrandom $N_\delta\in \mathbb N$ and $\Omega_\delta\subseteq 
\Omega_p$ with   $\mathbb P\{\Omega_\delta\}\ge p/2$,  such that, for all $n\ge N_\delta$, $\omega \in \Omega_\delta$
\[
x_n\in [0, \delta).
\]
By \eqref{eq:+infty}, for each $\omega\in \Omega_\delta$ there exists $n_\delta(\omega)\in \mathbb N$, $n_\delta\ge N_\delta$,  such that, on $\Omega_\delta$,
\[
\sum_{k=N_\delta}^{n_\delta(\omega)  } \sigma_k \xi_{k+1}(\omega)>\delta.
\]
For each $k\ge N_\delta$ and $\omega \in \Omega_\delta$, we have $x_k\in [0, \delta)\subset [0, K)$, so, by \eqref{ineq:stoch}, 
\begin{equation*}
x_{k+1}\ge f(x_k)+\sigma_k\xi_{k+1}\ge x_k+\sigma_k\xi_{k+1}\ge f(x_{k-1})+\sigma_{k-1}\xi_{k} +\sigma_k\xi_{k+1}.
\end{equation*}
Applying induction, we show that, on $\Omega_\delta$,
\[
x_{k+1}\ge \sum_{i= N_\delta}^{k}\sigma_i\xi_{i+1},
\]
which implies that, on $\Omega_\delta$,
\[
x_{n_\delta(\omega)+1}\ge \sum_{i=N_\delta}^{n_\delta(\omega)} \sigma_i \xi_{i+1}(\omega)>\delta.
\]
Obtained contradiction proves \eqref{cond:prob0}. 
\end{proof}

The following lemma is a corollary of Lemmata \ref{lem:+infty} and \ref{lem:infty1}.
\begin{lemma}
\label{lem:infty2}
Suppose that condition \eqref{cond:xnotzero} holds, 
and the sequence $(\xi_n)_{n\in \mathbb N}$ satisfies the assumptions of 
Lemma~\ref{lem:infty1}. Then \eqref{cond:prob0} holds for any solution $x_n$ to equation \eqref{eq:main} with the 
initial value $x_0>0$. 
\end{lemma}


\subsection{Case $\sigma\in {\bf \ell}_2$.}
\subsubsection{ Connection between noise intensity and $f$ which guarantees \eqref{cond:prob0}.}
\label{subsec:consigmaF}

\begin{lemma}
\label{lem:sigmaF}
Let Assumption \ref{as:chibound} and condition \eqref{cond:xnotzero} hold, $f$ be a nondecreasing function on 
$(0, \delta)$, for some $\delta \in (0,K)$, $(\sigma_n)_{n\in \mathbb N}$ be a non-increasing sequence and $\sigma\in {\bf \ell}_2$. Let $F$ be defined as in \eqref{def:F},
and $\xi_n$ either be identicaly distributed or satisfy condition~\eqref{eq:estprob} of Remark \ref{rem:noniid}.
Suppose also that  there exist $M\in (0, \infty)$ and $L\in {\mathbb N}$ such that  for all $n\ge L$ 
\begin{equation}
\label{cond:sigmaF}
\sigma_{n+1}\le F\bigl(M\sigma_{n}  \bigr).
\end{equation}
 Then any solution $x_n$ of equation \eqref{eq:main} with  an arbitrary   initial value $x_0>0$  satisfies \eqref{cond:prob0}.
\end{lemma}
\begin{proof}
Assume that
\[
\mathbb P [\overline \Omega]=p>0, \quad \overline \Omega:=\{\omega\in \Omega: \lim_{n\to \infty} x_n=0\}.
\]
Then, there is $\tilde \Omega\subseteq \overline \Omega$ and $\overline N\in \mathbb N$ such that 
$\mathbb P [\tilde \Omega]=p/2$ and, on $\tilde \Omega$ for all $n\ge \overline N$,
\[
x_n\in (0, \delta).
\]
Let us fix some $\varepsilon\in (0, 1)$ and denote 
\[
J=\left[\frac M{1-\varepsilon} \right].
\]
By Lemma \ref{lem:kpos}, with probability 1 there is an $N_J>\max\{\overline N, L\}$ 
such that $\xi_{i}\in (1-\varepsilon, 1)$ for all $i=N_J+1, N_J+2, \dots, N_J+J$.   
Therefore, on $\tilde \Omega$  all $x_n\in (0, \delta)$ for $n\ge N$ and $\lim\limits_{n\to \infty} x_n=0$. 

Thus, on $\tilde \Omega$,
\begin{equation*}
\begin{split}
x_{N+1}\ge &f(x_{N})+\sigma_{N}\xi_{N+1}\ge (1-\varepsilon)\sigma_N, \\
x_{N+2}\ge &f(x_{N+1})+\sigma_{N+1}\xi_{N+2}\ge x_{N+1}+(1-\varepsilon)\sigma_{N+1}> (1-\varepsilon)\sigma_N+ (1-\varepsilon)\sigma_{N+1}, \\ 
\text{and, inductively,} &\\
x_{N+J}=&f(x_{N+J-1})+\sigma_{N+J-1}\xi_{N+J}\ge x_{N+J-1}+(1-\varepsilon)\sigma_{N+J-1}\ge (1-\varepsilon)\sum_{i=0}^{i=J-1}\sigma_{N+i}.
\end{split}
\end{equation*}
By monotonicity of $\sigma_n$ and by the definition of $J$, we have
\[
(1-\varepsilon)\sum_{i=0}^{i=J-1}\sigma_{N+i}\ge (1-\varepsilon)J\sigma_{N+J-1}\ge M \sigma_{N+J-1}.
\]
Then, by conditions  \eqref{cond:xnotzero} and \eqref{cond:sigmaF}, we have, on $\tilde \Omega$,
\begin{equation*}
\begin{split}
x_{N+J+1}&\ge  f(x_{N+J})+\sigma_{N+J}\xi_{N+J+1}\ge  f(M \sigma_{N+J-1}  )-\sigma_{N+J} \\
& =  F(M \sigma_{N+J-1})+M \sigma_{N+J-1}-\sigma_{N+J}  \ge M \sigma_{N+J-1};\\ 
\text{and, inductively, for each} & \,\,   k\in \mathbf N,  \\
x_{N+J+k}&\ge  f(M \sigma_{N+J-1})-\sigma_{N+J+k}
=  F(M \sigma_{N+J-1})+M \sigma_{N+J-1}-\sigma_{N+J+k} \\
& \geq \sigma_{N+J}+M \sigma_{N+J-1}- \sigma_{N+J+k} \ge M \sigma_{N+J-1}.
\end{split}
\end{equation*}
Obtained contradiction  proves the result. 
\end{proof}

\begin{remark}
We can get rid of the assumption of monotonicity of $\sigma_n$ assuming instead, that 
for some $\varepsilon >0$, $J\in \mathbb N$, $N\in \mathbb N$ and each $k\in \mathbb N$, 
\[
\sigma_{N+J+k}\le F\left((1-\varepsilon)  \sum_{i=0}^{i=J-1}\sigma_{N+i}\right).
\]
\end{remark}

\medskip

\subsubsection{Case $\liminf\limits_{x\to 0+}\frac{f(x)}x>1$.}

In this subsection we assume that in addition to  \eqref{cond:xnotzero}  the following condition  holds:
\begin{equation}
\label{cond:fkappa_1}
\liminf_{x\to 0+}\frac{f(x)}x>1.
\end{equation}

Note that condition \eqref{cond:fkappa_1}  implies that 
\begin{equation}
\label{cond:fkappa_2}
\text{ there exist $\delta, \kappa >0$ such that  $f(x)\ge (1+\kappa)x$ for $x\in (0, \delta)$}.
\end{equation}

\begin{lemma}
\label{lem:1+kappa}
Let Assumption \ref{as:chibound}, conditions \eqref{cond:xnotzero} and \eqref{cond:fkappa_1}  hold, $(\sigma_n)_{n\in \mathbb N}$ be a non-increasing sequence and $\sigma\in {\bf \ell}_2$.  Then for any 
solution $x_n$ of equation \eqref{eq:main} condition \eqref{cond:prob0} holds.
\end{lemma}

\begin{proof}
Let $\kappa$ be defined as in \eqref{cond:fkappa_2}.  Then condition \eqref{cond:sigmaF}  holds for $M>\frac 1\kappa$.
Indeed,
\[
F(M\sigma_n)=f(M\sigma_n)-M\sigma_n\ge \kappa M\sigma_n\ge \sigma_n\ge \sigma_{n+1}.
\]
Reference to Lemma \ref{lem:sigmaF} completes the proof.

\end{proof}

\begin{remark}
\label{rem:sigmaF}
Condition \eqref{cond:sigmaF} can be satisfied while \eqref{cond:fkappa_1}  does not hold. 
In particular, it can happen when $\sigma_n$ decays very quickly.  Let 
\[
x^2<f(x)-x<2x^2, \quad x\in (0, \delta),
\]
and
\[
\sigma_n=q^{2^n}, \quad q\in (0, 1).
\]
Then  $\liminf\limits_{x\to 0+}\frac{f(x)}x=1$,  so $f$ does not satisfy \eqref{cond:fkappa_2}. But 
\[
\sigma_{n+1}=q^{2^{n+1}}=(q^{2^n})^2 <F(q^{2^n})=F(\sigma_n),
\]
so \eqref{cond:sigmaF} holds with $M=1$.
\end{remark}

\medskip

\subsubsection{Symmetric noises.}
\label{subsec:gencond}

Suppose that Assumption \ref{as:chibound} holds, $\sigma \in {\bf \ell}_2$, ${\mathbb E}(\xi_i)=0$.  By  Theorem \ref{thm:2ser},  
when   $\sigma\in {\bf \ell}_2$,  for each $N\in \mathbb N$, the sum
\begin{equation}
\label{def:TN(n)}
T_N(n):=\sum_{i=N}^{n} \sigma_i\xi_{i+1}, \quad n\ge N,
\end{equation}
 converges a.s. to a random variable $T_N$,
\begin{equation}
\label{def:TN}
T_N:=\sum_{i=N}^{\infty} \sigma_i\xi_{i+1},
\end{equation}
which has a zero mean, $\mathbb E(T_N)=0$, and the variance $\displaystyle \mathbb Var (T_N)=\sum_{i=N}^{\infty} \sigma_i^2.$

The proof of the next lemma is deferred to Appendix.
\begin{lemma}
\label{lem:symTN}
Let $\sigma\in {\bf \ell}_2$, Assumption \ref{as:chibound} hold and 
\begin{equation}
\label{as:sym}
\begin{split}
&\xi_i, \, \, i\in \mathbf N,  \,\,  \text{are independent identically distributed random variables } \\& \text{with the density $\phi$ being even: $\phi(x)=\phi(-x)$ for all $x\in [-1, 1].$}
\end{split}
\end{equation}
Let $T_N$ be defined as in \eqref{def:TN}.  Then, for each $N\in \mathbb N$,  
\begin{equation}
\label{prop:symTN}
\mathbb P[T_N\le 0]=\frac 12.
\end{equation}
\end{lemma}
Note that \eqref{as:sym} implies ${\mathbb E} \xi_i=0$.

\begin{lemma}
\label{lem:<12}
Let Assumptions \ref{as:g12}, \ref{as:chibound} and condition \eqref{as:sym} hold, $\sigma\in {\bf \ell}_2$, and $x_n$ be a 
solution to equation \eqref{eq:main} with the initial value $x_0>0$. Then 
\begin{equation}
\label{cond:prob1/2}
\mathbb P\{\omega\in \Omega: \lim_{n\to \infty} x_n(\omega)=0\}\le 1/2.
\end{equation}
\end{lemma}
\begin{proof}
Define
\[
A:=\{\omega\in \Omega: \lim_{n\to \infty} x_n(\omega)=0\}.
\]
By Lemma \ref{lem:lim0}, for each $\alpha \in (0,{\mathbb P}\{A\})$, we can find nonrandom number $N_\alpha\in \mathbb N$ and the  set $A_\alpha\subseteq A$ such that 
\[
x_n(\omega) \in [0, \delta), \quad n\ge N_\alpha, \quad \omega\in A_\alpha, \quad \mathbb P\{A_\alpha\}>\mathbb P\{A\}-\alpha>0.
\]
Our purpose is to show that, for each $\alpha>0$,
\[
\mathbb P[A_\alpha]\le 1/2.
\]
Since $\alpha$ can be arbitrarily small, this will imply $\mathbb P[A]\le 1/2$.
Applying  the same estimations as in  \eqref{ineq:stoch}, we obtain that, for  $n\ge N_\alpha$, $\omega\in A_\alpha$,
\begin{equation}
\label{est:xNa}
x_{n+1}\ge \sum_{i=N_\alpha}^n\sigma_{i}\xi_{i+1}.
\end{equation}
Since   $\displaystyle \sum_{i=N_\alpha}^n\sigma_{i}\xi_{i+1}$ converges a.s. and, for  $\omega\in A_\alpha$, 
we have $x_n\to 0$,  estimate \eqref{est:xNa} yields that on $A_\alpha$,
$$
0\ge \sum_{i=N}^\infty \sigma_{i}\xi_{i+1}=T_N.
$$
So, $A_\alpha\subseteq \{T_N\le 0\}$,  which implies 
$\displaystyle 
\mathbb P[A_\alpha]\le \mathbb P[\Omega_{N_\alpha}]$.
Now the result follows from \eqref{prop:symTN}.
\end{proof}


\subsection{Main stochastic theorem}
\label{subsec:genstochthm} 
In Theorem \ref{thm:mainstoch} below we present  conditions when a solution of equation \eqref{eq:main} converges to a positive 
equilibrium and estimate the probability of this event. 
The proof is a corollary of Theorem \ref{thm:detmain} and  Lemmata 
\ref{lem:infty2}, \ref{lem:sigmaF}, \ref{lem:symTN}.

\begin{theorem}
\label{thm:mainstoch}
Let Assumptions \ref{as:g12}, \ref{as:g34}, \ref{as:fxinf}, \ref{as:chibound}  and condition \eqref{cond:sigmatozero} hold,
$x_n$ be a solution to equation \eqref{eq:main} with the initial value $x_0>0$. 
\begin{enumerate}
\item [(i)] 
If $\sigma\notin {\bf \ell}_2$, $\mathbb E\xi_n=\mu_n$,  $\mathbb E\xi_n^2=1$, $\mu^-=(\mu_n^-)\in {\bf \ell}_2$ then 
\begin{equation}
\label{rel:K}
\mathbb P\{\omega\in \Omega: \lim_{n\to \infty} x_n(\omega)=K\}=1.
\end{equation}
\item [(ii)] Let $\sigma\in {\bf \ell}_2$.
\begin{enumerate}
\item If condition \eqref{cond:sigmaF}  is fulfilled, $f$ is a nondecreasing function on  $(0, \delta)$ for some $\delta>0$, $(\sigma_n)_{n\in \mathbb N}$ is a non-increasing sequence, $\xi_n$ either are identically distributed or satisfy Remark \ref{rem:noniid}, then \eqref{rel:K} holds.
\item If condition\eqref{as:sym} is fulfilled then $ \mathbb P\{\omega\in \Omega: \lim_{n\to \infty} x_n(\omega)=K\}>1/2.$
\end{enumerate}
\end{enumerate}
\end{theorem}


\section{Simulations}
\label{sec:sim}

In this section we simulate solutions of the stochastically perturbed equations
\begin{equation}
\label{ex1_eq1}
x_{n+1}=\max\left\{ x_n+x_n^2-x_n^3-x_n^4+ \sigma_n \xi_{n+1}, 0 \right\}, \quad n=0, 1, \dots, 
\end{equation}
where $\xi$ are uniformly distributed in $[-1,1]$ and $\sigma_n$ takes 3 forms: it is either zero (which corresponds to the 
deterministic case), or  $\sigma_{n}=\varepsilon/n^d$, $d>0$, or $\sigma_n=e^{-2^n}$.
We are mostly interested in small initial values since we want to investigate the behavior of solution in the right neighborhood of zero. 

Fig.~\ref{figure1} shows convergence of a solution to the equilibrium $K \approx 0.618034$ in the deterministic case,  $\sigma_{n}\equiv 0$, for a  small initial value $x_0=0.001$.

\begin{figure}[ht]
\centering
\includegraphics[height=.22\textheight]{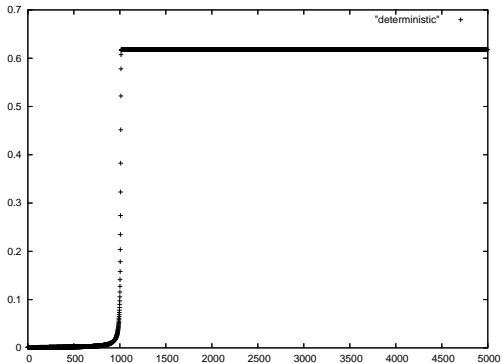}
\caption{The solution of the deterministic equation with $x_0=0.001$}
\label{figure1}
\end{figure}

Consider the function 
$$f(x)=(x+x^2)(1-x^2)=x+x^2-x^3-x^4$$  
which satisfies
$$f(x)>x,\,\, x \in (0,K), \,\, f(x)<x, \,\, x > K, \,\, K \approx 0.618034.$$
The derivative
$$f'(x)=1+2x-3x^2-4x^3>0, ~~x \in [0,K],$$
so the equilibrium is globally asymptotically stable, with eventually monotone convergence to $K$.
In fact, the derivative vanishes at the unique positive point $\approx 0.6404>K$, so $f$ is a unimodal function. 
Next, 
$$F(x)=f(x)-x=x^2-x^3-x^4,$$
and condition \eqref{cond:sigmaF} has the form
\begin{equation}
\label{cond:40}
\sigma_{n+1} \leq M^2\sigma_{n}^2-M^3\sigma_{n}^3-M^4\sigma_{n}^4, \,\,  n \geq \bar N,
\end{equation}
which is impossible for any constant $M>0$ and $\sigma_{n}=\varepsilon/n^d$, $d>0$.
Thus, for $\sigma_{n}=\varepsilon/n^d$, $d>0$, condition \eqref{cond:sigmaF} does not hold.

Fig.~\ref{figure2} presents asymptotics of solutions of the stochastically perturbed equations \eqref{ex1_eq1} with 
$\sigma_{n}=\varepsilon/n^d$, $d>0$. 
The value  $d=0.5$ corresponds to the case when $\sigma\notin {\bf \ell}_2$, and the left part of the 
Fig.~\ref{figure2} demonstrates that all solutions converge to the equilibrium $K \approx 0.618034$.  Value 
$d=8$ corresponds to the case when $\sigma\in {\bf \ell}_2$, and the right part of Fig.~\ref{figure2} demonstrates that approximately 
a half of solutions converges to $K$ while another half converges to zero.


\begin{figure}[ht]
\centering
\includegraphics[height=.22\textheight]{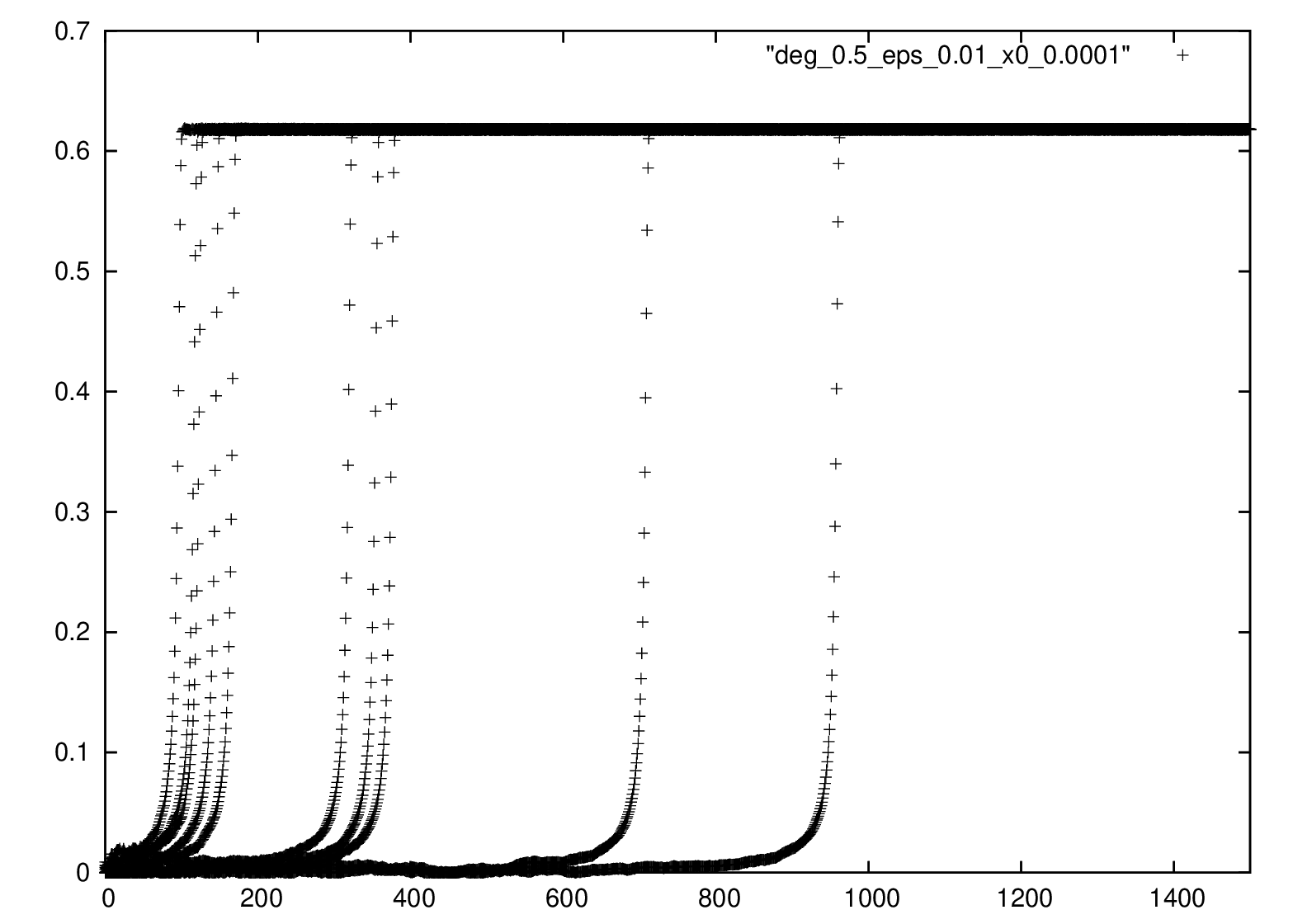}
\includegraphics[height=.22\textheight]{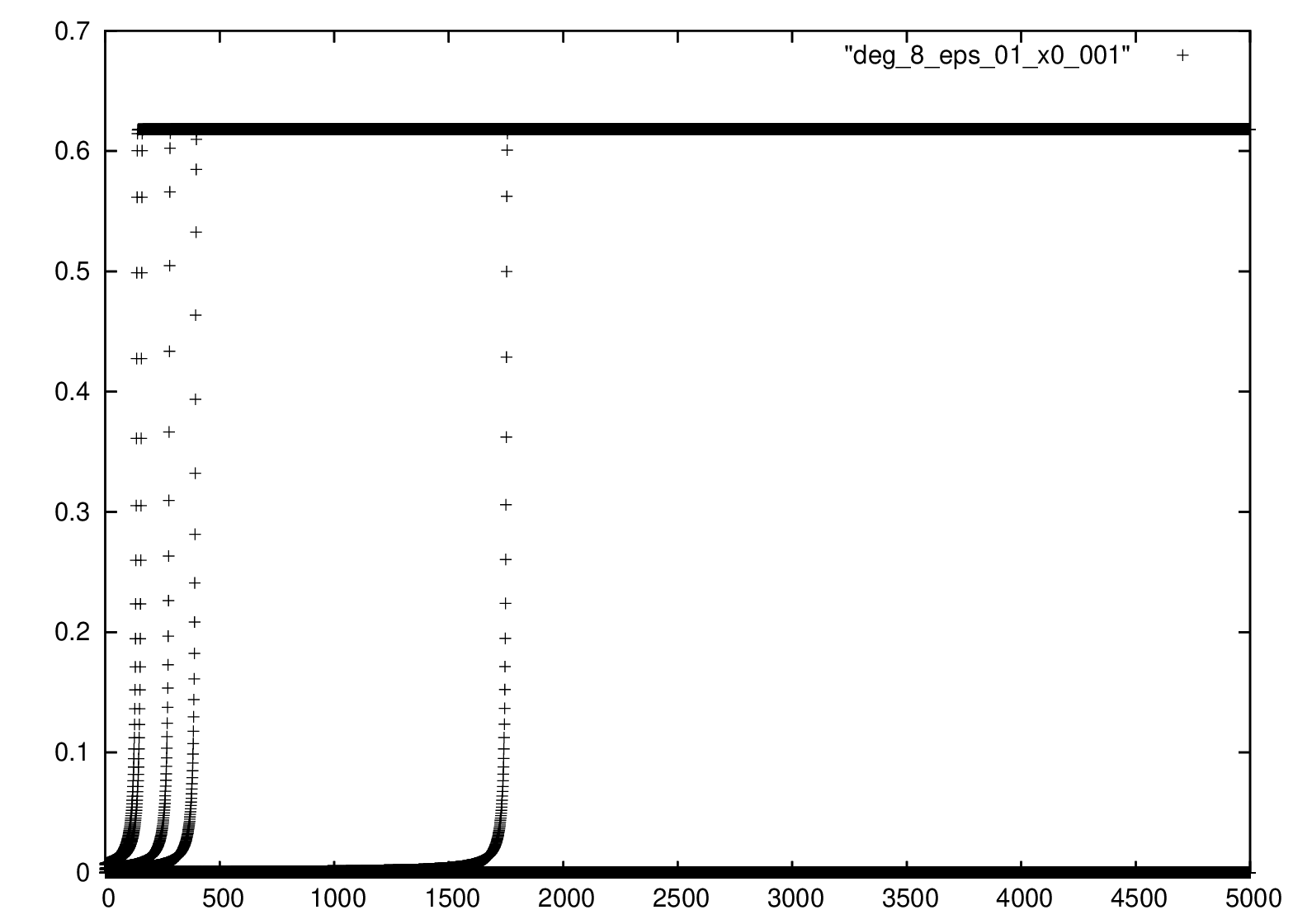}
\caption{Ten runs of equation (\protect{\ref{ex1_eq1}}) with $x_0=0.001$, $\varepsilon=0.01$ and $d=0.5$ (left), $d=8$ 
(right), $\xi$ are uniformly distributed in $[-1,1]$.}
\label{figure2}
\end{figure}

However, for 
\[
\sigma_n=e^{-2^n}, 
\]
condition \eqref{cond:40}  hold, for example, with $M=2$. To show that we note that there exists $N_1\in \mathbf N$ such that, for 
all $n\ge N_1$,  we have
\[
1-2e^{- 2^n} -4e^{-2 \cdot 2^n}\ge \frac 12.
\]
So, for all $n\ge N_1$, 
\[
4e^{-2 \cdot 2^n}-8e^{-3 \cdot 2^n}-16e^{-4\cdot 2^n}=4e^{-2 \cdot 2^n}\left[ 1-2e^{- 2^n} -4e^{-2 \cdot 2^n}\right]\ge 2e^{-2 \cdot 2^n}>e^{-2^{n+1}},
\]
which implies \eqref{cond:40}.


The case when the noise is decaying faster than the geometric sequence, $\sigma_n=e^{-2^n}$,
is illustrated in Fig.~\ref{figure3}: all solutions converge to the equilibrium $K \approx 0.618034$. 

\begin{figure}[ht]
\centering
\includegraphics[height=.22\textheight]{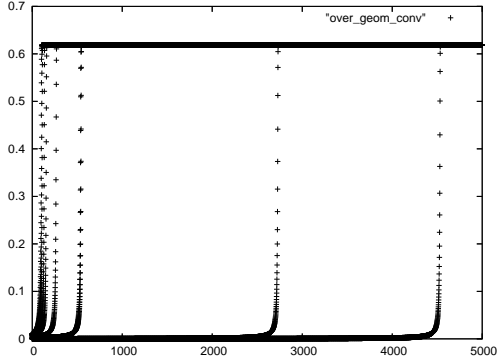}
\caption{Ten runs with $x_0=0.001$, the noise $e^{-2^n} \xi_n$, where $\xi$ are uniformly distributed in $[-1,1]$.}
\label{figure3}
\end{figure}


\section{Discussion}

For one-dimensional maps describing models of population dynamics, there are several methods to stabilize an 
equilibrium, see, for example, \cite{BF15,BL} and references therein. 
In the present paper, we considered the case of the unique 
positive equilibrium, though the technique can be applied to problems with several positive equilibria. 
The phenomenon that the solution under stochastic (and even deterministic) eventually vanishing perturbations can tend to zero may be 
compared to the Allee effect. For each map and maximum amplitude of perturbations,
there is a certain threshold. If the initial value exceeds the threshold, it tends to the positive equilibrium; 
with lower initial values, solutions may tend to zero. However, for unbounded noises $\xi_i$, for example, with the
normal distribution $\mathcal N(0,1)$, this is no longer valid. In this case, no initial value can guarantee that there is no $x_n=0$ 
in the solution sequence.

Stochastic perturbations frequently occur in nature, so
their influence on persistence of populations is an important question. 
In the present paper we considered the case when perturbations are bounded
and tend to zero with time, and still there is a harmful effect 
of stochastic perturbations on population survival. There are still many problems to be explored.

\begin{itemize}
\item
Can the estimate of $1/2$ in Theorem \ref{thm:mainstoch} (b) be improved? 
\item
For unbounded noises $\xi_i$, like $\mathcal N(0, 1)$,  we cannot expect to get results in (i) and (ii)(a) of
Theorem \ref{thm:mainstoch} to hold a.s. It would be interesting to estimate corresponding probabilities.
To update the proofs, it would probably be necessary to impose extra assumption on $\sigma_n$
in order to insure that the Lindeberg condition (and therefore the Central Limit Theorem) holds. 
\item
In the case of multiple positive equilibrium points, in which cases will the probability that a solution of a stochastically 
perturbed equation with an eventually vanishing perturbation tends to the originally unstable equilibrium will be greater than zero? 
\item
If there is a unique unstable positive equilibrium combined with a stable cycle, can the results of the present paper be extended
to establish conditions and probabilities that the solution subject to stochastic perturbations will converge to this cycle?
Two-cyclic behaviour of difference equations subject to eventually vanishing stochastic perturbations was studied in \cite{BRCAM12}.
\end{itemize}

\section{Acknowledgment} 

The authors are grateful to Professor J.\,A.\,D.  Appleby for fruitful discussions. 

The research of both authors was supported by AIM SQuaRE program and by NSERC grant RGPIN-2015-05976.

\section{Appendix}
\label{sec:ap}

\subsection{Proof of Lemma \ref{lem:lim0}}

Note that for each $\omega \in \bar \Omega$, there exists $N(\omega)$ such that for all $n\ge N(\omega)$,
we have $x_n(\omega)\in [0, \delta)$.
Define
\[
\Omega_i:=\{\omega\in \bar \Omega:\,\,  i\le N(\omega)<i+1\}.
\]
Then 
$\displaystyle 
\bar \Omega=\bigcup_{i=1}^\infty\Omega_i, \quad \Omega_i\cap \Omega_j=\emptyset, \quad i\neq j$,
and
$\displaystyle
p=\mathbb P\{\bar \Omega)\}=\mathbb P\left\{\bigcup_{i=1}^\infty\Omega_i\right\}=\sum_{i=1}^\infty\mathbb P\{\Omega_i\}$.
Thus there exists $\bar N=\bar N(\delta, \alpha)\in \mathbf N$ such that
$\displaystyle
\sum_{i=1}^{\bar N}\mathbb P\{\Omega_i\}\ge p-\alpha$.
Denote 
\[
\Omega(\delta, \alpha):=\bigcup_{i=1}^{\bar N}\Omega_i=\{\omega\in \bar \Omega:  N(\omega)\le \bar N\}=\{\omega\in \bar \Omega: x_n(\omega)<\delta, \,\, n\ge \bar N\}.
\]
So, 
\[
\mathbb P\{\Omega(\delta, \alpha)\}\ge p-\alpha, \quad \text{and} \quad x_n(\omega)\in [0,\delta)  \quad \text{whenever} \quad \omega \in \Omega_\delta, n\ge \bar N.
\]

\medskip

\subsection{Proof of Lemma \ref{lem:infty1}}
\label{subsec:CLT}

The proof of Lemma \ref{lem:infty1} can be obtained as a corollary of Theorems \ref{thm:CLTL},  
Lemma \ref{lem:sqrtn} and Remark \ref{rem:Lind}  which are further stated in this section.

The Central Limit Theorem for (normalized and centralized) sum of independent 
random variables $\zeta_1, \zeta_2, \dots, \zeta_n, \dots$ is proved in \cite[p. 328--332]{Shiryaev96}, under the assumption that the 
classical {\it Lindeberg condition} is satisfied.  We formulate this theorem below.

\begin{theorem}[Central Limit Theorem]
 \label{thm:CLTL}
  Let  $\zeta_1, \zeta_2, \dots, \zeta_n, \dots$ be a sequence of  independent random variables
  with a finite second moment. Let $\mathbb E\zeta_k=0$, $\mathbb E\zeta^2_k=\sigma_k^2$, $S_n=\zeta_1+ \zeta_2+ \dots+ \zeta_n$, 
$D^2_{n}:=\sum_{k=1}^{n}\sigma^2_k$, and let $F_k=F_k(x)$ be a distribution function of the random variable $\zeta_k$. 

Suppose that $\displaystyle \lim_{n \to \infty} D^2_{n} = \infty$, and the Lindeberg condition is satisfied, i.e. for every 
$\varepsilon>0$
\begin{equation}
\label{cond:Lind}
\lim_{n \to \infty} \frac 1{D^2_{n}}\sum_{k=1}^{n}\int\limits_{x:|x|\ge \varepsilon D_n}  \!\!\!\!\!\!\!\!\!\ xdF_k(x) = 0.
\end{equation}

Then
\begin{equation*}
\frac{S_n}{D_n}\,\,  \xrightarrow{d} \,\,\mathcal N(0, 1).
\end{equation*}
\end{theorem}

By applying the Central Limit Theorem \ref{thm:CLTL}, the  following theorem was proven in \cite{AR1}.
\begin{lemma}
 \label{lem:sqrtn}
 Let assumptions of Theorem \ref{thm:CLTL} holds.  Then
\begin{equation}
\label{cond:sqrt}
 \limsup_{n\to \infty}\frac{S_n}{D_n}=\infty,\quad
 \liminf_{n\to \infty}\frac{S_n}{D_n}=-\infty, \quad\text{a.s.}
\end{equation}
\end{lemma}

\begin{remark}
\label{rem:Lind}
It was proved in \cite[p. 332--333]{Shiryaev96} that the Lindeberg condition \eqref{cond:Lind} holds if  $D^2_{n}\to 
\infty$ as 
$n\to \infty$ and $\zeta_n$ are uniformly bounded, i.e. $|\zeta_n|\le \bar M$ for some $\bar M\in \mathbb R_+$ and all $n\in \mathbb N$.
\end{remark}

Applying the above results to $\zeta_n=\sigma_{n-1}\xi_n$,  we conclude that when $\xi_n$ satisfies Assumption \ref{as:chibound},  
$\mathbb E\xi_n=0$,  $\mathbb E\xi_n^2=1$ and   $\sigma \notin {\bf \ell}_2$ we have
\begin{equation}
\label{rel:infty}
\limsup_{n\to \infty}\frac{\sum_{i=0}^n \sigma_i\xi_{i+1}}{\sqrt{\sum_{k=0}^{n}\sigma^2_k}}=\infty,\quad
 \liminf_{n\to \infty}\frac{\sum_{i=0}^n \sigma_i\xi_{i+1}}{\sqrt{\sum_{k=0}^{n}\sigma^2_k}}=-\infty, \quad\text{a.s.},
\end{equation}
which implies condition \eqref{eq:+infty}.

Assume now that   $\mathbb E\xi_n=\mu_n$ and   $\mu^-=(\mu_n^-) \in {\bf \ell}_2$. Applying the above results to 
$\zeta_n=\sigma_{n-1}(\xi_n-\mu_n)$, we arrive at 
\[
\limsup_{n\to \infty}\frac{\sum_{i=0}^n \sigma_i[\xi_{i+1}-\mu_{i+1}]}{\sqrt{\sum_{k=0}^{n}\sigma^2_k}}=\infty,
\]
or
\begin{equation}
\label{rel:inftymean}
\limsup_{n\to \infty}\left[\frac{\sum_{i=0}^n \sigma_i\xi_{i+1}}{\sqrt{\sum_{k=0}^{n}\sigma^2_k}}-
\frac{\sum_{i=0}^n \sigma_i\mu_{i+1}}{\sqrt{\sum_{k=0}^{n}\sigma^2_k}}\right]=\infty.
\end{equation}
Since $\mu_i=\mu_i^{-}+\mu_i^+\ge -[-\mu_i^{-}]$, we have
\[
\frac{\sum_{i=1}^n \sigma_i \mu_{i+1}}{\sqrt{\sum_{k=1}^{n}\sigma^2_k}} \ge -\frac{\sum_{i=1}^n \sigma_i[-\mu_{i+1}^-]}{\sqrt{\sum_{k=1}^{n}\sigma^2_k}}. 
\]
Applying the H\"older inequality we obtain, for each $n\in \mathbb N$,
\[
\frac{\sum_{i=0}^n \sigma_i [-\mu_{i+1}^-]}{\sqrt{\sum_{k=0}^{n}\sigma^2_k}}
\le \frac{\sqrt{\sum_{k=0}^{n}\sigma^2_k}\sqrt{\sum_{i=0}^n [\mu^-_{i+1}]^2}}{\sqrt{\sum_{k=0}^{n}\sigma^2_k}}=\sqrt{\sum_{i=1}^n 
[\mu^-_{i+1}]^2}\le \|\mu^-\|_{{\bf \ell}_2}.
\]
Then 
\begin{eqnarray*}
 \frac{\sum_{i=1}^n \sigma_i\xi_{i+1}}{\sqrt{\sum_{k=0}^{n}\sigma^2_k}}&= &\frac{\sum_{i=1}^n 
\sigma_i\xi_{i+1}}{\sqrt{\sum_{k=1}^{n}\sigma^2_k}}-
\frac{\sum_{i=1}^n \sigma_i\mu_{i+1}}{\sqrt{\sum_{k=1}^{n}\sigma^2_k}}+ 
\frac{\sum_{i=1}^n \sigma_i\mu_{i+1}}{\sqrt{\sum_{k=1}^{n}\sigma^2_k}} \\
& \ge & \frac{\sum_{i=1}^n \sigma_i\xi_{i+1}}{\sqrt{\sum_{k=1}^{n}\sigma^2_k}}-
\frac{\sum_{i=1}^n \sigma_i\mu_{i+1}}{\sqrt{\sum_{k=1}^{n}\sigma^2_k}}-\frac{\sum_{i=0}^n \sigma_i [-\mu_{i+1}^-]}{\sqrt{\sum_{k=0}^{n}\sigma^2_k}}\\
& \ge & \frac{\sum_{i=1}^n \sigma_i\xi_{i+1}}{\sqrt{\sum_{k=1}^{n}\sigma^2_k}}-
\frac{\sum_{i=1}^n \sigma_i\mu_{i+1}}{\sqrt{\sum_{k=1}^{n}\sigma^2_k}}-\|\mu^-\|_{{\bf \ell}_2}.
\end{eqnarray*}
Thus \eqref{rel:inftymean} implies \eqref{eq:+infty}.


\subsection{Proof of Lemma \ref{lem:symTN}}
\label{subsec:symTN}

Lemma \ref{lem:symTN} is a corollary of the following Lemma.
\begin{lemma}
\label{lem:symSn}
Let Assumption \ref{as:chibound} and condition \eqref{as:sym} hold and $\sigma\in {\bf \ell}_2$.  Let, for each $N\in \mathbb N$ and 
$n\ge N$, random variables  $T_N(n)$ and $T_N$ be defined by 
\eqref{def:TN(n)} and \eqref{def:TN} respectively. Then, for each $x\in \mathbb R$, $N\in \mathbb N$ and $n\ge N$,  we have
\begin{enumerate}
\item [(i)] $\mathbb P[T_N(n)=x]=0$, \, $\mathbb P[T_N=x]=0$;
\item [(ii)] a) $\mathbb P[T_N(n)\le x]=\mathbb P[T_N(n)\ge - x]$; \, \\ b) $\mathbb P[T_N\le x]=\mathbb P[T_N \ge - x]$.
\end{enumerate}
\end{lemma}

\begin{proof}
(i) The result stated in (i) will follow from the continuity of distributions of $T_N(n)$ and $T_N$. 
Denote for simplicity
\begin{equation}
\label{def:zi}
z_i:=\sigma_{i-1}\xi_i, \quad i\in \mathbb N.
\end{equation}
Then
\[
T_{N}=\sum_{i=N}^{\infty} z_i=\sum_{i=N+1}^{\infty} z_i+z_{N}=T_{N+1}+z_{N},
\]
and $z_N$ is independent of $T_{N+1}$.
Let $F_{N+1}$ be the probability distribution function for $T_{N+1}$.
Then, for $x\in \mathbb R$, 
\begin{equation}
\begin{split}
&\mathbb P[T_{N}=x]=\mathbb P[T_{N+1}+z_{N}=x]=\int_{-\infty}^\infty \mathbb P\left[T_{N+1}+z_{N}=x\biggl| T_{N+1}=y\right]dF_{N+1}(y)\\
&=\int_{-\infty}^\infty \mathbb P\left[z_{N}=x-y\biggl| T_{N+1}=y\right]dF_{N+1}(y)=\int_{-\infty}^\infty \mathbb P\left[z_{N}=x-y\right]dF_{N+1}(y)=0,
\end{split}
\end{equation}
since for the continuous random variable $z_{N}$  we have $\mathbb P\left[z_{N}=x-y\right]=0$ for each $x, y\in \mathbb N$.

The same approach is applied in the proof of $\mathbb P[T_N(n)=x]=0$.

\medskip

(ii), a). Prove first that, for each $N\in \mathbb N$,
\[
\mathbb P [T_N(N+1)\le x]=\mathbb P[T_N(N+1)\ge -x].
\]
Random variables $z_i$ and $z_j$, $i\neq j$, defined as in \eqref{def:zi},  are continuous and independent. 
 By Assumption \ref{as:chibound}  and condition \eqref{as:sym} each $\xi_i$ has  a density function $\phi$ such that $\phi(-x)=\phi(x)$. Since $\sigma_n>0$ for each $n\in \mathbf N$, we conclude that $z_i$ has  a density function $f_i$, defined by 
  \[
 f_i(x)=\frac 1{\sigma_{i-1}}\phi\left( \frac x{\sigma_{i-1} }\right), \quad x\in \mathbb R,
 \]
  and  $f_i(-x)=f_i(x)$.  
  
 If we denote by $f_{-i}$ a  density function for $-z_i$, then
  \[
  f_{-i}(x)=f_{i}(-x)= f_i(x), \quad \forall x\in \mathbb R.
   \]
Now, on the one hand, 
\begin{equation*}
\begin{split}
&\mathbb P [ T_N(N+1) \le x]=\mathbb P [z_{N}+z_{N+1}\le x]=\int_{-\infty}^\infty \mathbb P\left[z_{N}\le x-y\biggl| z_{N+1}=y\right]f_{N+1}(y)dy\\\\
&\int_{-\infty}^\infty \mathbb P\left[z_{N}\le x-y \right]f_{N+1}(y)dy=\int_{-\infty}^\infty \left[\int_{-\infty}^{x-y}f_{N}(t)dt\right]f_{N+1}(y)dy.
\end{split}
\end{equation*}


On the other hand,
\begin{equation*}
\begin{split}
\mathbb P [ T_N(N+1)&\ge -x]=\mathbb P [- T_N(N+1)\le x]=\mathbb P [-z_{N}-z_{N+1}\le x]\\\\&
=\int_{-\infty}^\infty \left[\int_{-\infty}^{x-y}f_{-N}(t)dt\right]f_{-(N+1)}(y)dy=\int_{-\infty}^\infty \left[\int_{-\infty}^{x-y}f_{N}(t)dt\right]f_{N+1}(y)dy\\\\&=\mathbb P [T_N(N+1)\le x].
\end{split}
\end{equation*}

\medskip

Presenting  $ T_N(n)= T_{N-1}(n)+z_n$  for $n\ge N+2$, and applying the mathematical induction we prove similarly that 
\begin{equation}
\label{eq:SnN}
\mathbb P \left[T_N(n)\le x \right]=\mathbb P \left[T_N(n)\ge - x \right].
\end{equation}

\medskip

(ii), b) {\it Convergence of distributions}. Let us prove that, for all $x\in \mathbb R$, 
\[
\mathbb P \left[T_N\le x \right]=\mathbb P \left[T_N\ge  -x \right].
\]
Theorem  \ref{thm:2ser} implies that  that $T_N(n)\to T_N$, a.s., so $T_N(n)$ converges to  $T_N$ in distribution and thus
\[
\lim_{n\to \infty} \mathbb P \left[ T_N(n)\le x \right]=\mathbb P \left[ T_N \le x\right].
\]
Similarly,
\[
\lim_{n\to \infty} \mathbb P \left[ T_N(n)\ge - x \right]=
\lim_{n\to \infty} \left\{1-\mathbb P \left[T_N(n)\le -x \right]\right\}=1-\mathbb P \left[T_N\le -x\right]=\mathbb P \left[ T_N \ge  -x\right].
\]
Proceeding to the limits in \eqref{eq:SnN} yields
\begin{equation}
\label{rel:1}
\mathbb P \left[T_N\le x \right]=\mathbb P \left[T_N\ge  -x \right].
\end{equation}
\end{proof}

To complete the proof of Lemma \ref{lem:symTN}  we take $x=0$ in \eqref{rel:1} which gives 
\[
\mathbb P \left[T_N\le 0 \right]=\mathbb P \left[T_N\ge 0 \right]=1-\mathbb P \left[T_N\le 0 \right]
\]
and  implies the necessary result:
\[
\mathbb P \left[T_N\le 0 \right]=\frac{1}{2}.
\]

\end{document}